\documentclass[onefignum,onetabnum]{siamart190516}
\usepackage{amsmath,amssymb}                
\usepackage{bbm}
\usepackage{graphicx}
\usepackage{mathrsfs}
\usepackage{amssymb}
\usepackage{epstopdf}
\usepackage{tikz-cd}
\usepackage{changes}
\usepackage{enumerate}
\usepackage{subfig}  
\usepackage{algorithmic}
\usepackage{longtable}
\usepackage{multirow}
\usepackage{array}

\usepackage{MnSymbol}
\allowdisplaybreaks

\newtheorem{Lemma}{Lemma}[section]
\newtheorem{Theorem}{Theorem}

\newtheorem{Corollary}[Lemma]{Corollary}
\newtheorem{Remark}[Lemma]{Remark}
\newtheorem{Definition}[Lemma]{Definition}
\newtheorem{Hypothesis}[Lemma]{Hypothesis}

\newcommand{\R}{\mathbb{R}}
\newcommand{\C}{\mathbb{C}}
\newcommand{\N}{\mathbb{N}}
\newcommand{\Z}{\mathbb{Z}}

\newcommand{\eps}{\varepsilon}

\def\Im{\mathop{\mathrm{Im}}}

\newcommand{\rmO}{\mathrm{O}}

\title{ The role of boundary constraints in simulating a nonlocal Gray-Scott model
 \thanks{This work is supported in part by the National Science
  Foundation under grant DMS-1911742 (GJ).} 
 }

\author{Loic Cappanera\thanks{Department of Mathematics, 
University of Houston, Houston, TX
(\email{lmcappan@central.uh.edu}).}
\and Gabriela Jaramillo\thanks{Department of Mathematics, University of Houston, Houston, TX (corresponding author, \email{gabriela@math.uh.edu})}
}

\begin{document}
\emergencystretch 3em


\maketitle

\begin{abstract}
{
We present a second-order algorithm for approximating solutions to nonlocal diffusive processes in reaction-diffusion equations. The numerical scheme relies on a quadrature method for the spatial discretization and a second-order Adams-Bashford method for the time marching. This algorithm is then used to simulate a nonlocal Gray-Scott model, known for generating interesting structures including periodic patterns, traveling waves, pulse and multi-pulse solutions. Our main goal is to study the impact of boundary constraints on the formation of stationary pulse solutions. We consider nonlocal Dirichlet and Neumann boundary constraints, as well as what we refer to as `free' boundary conditions. In addition, we investigate the effects of using different convolution kernels, fat- or thin-tailed, on the formation of these localized solutions. Our numerical results show that when the spread of the kernel is large, i.e. when the model is nonlocal, both the type of kernel and the type of boundary constraint used have a strong impact on the solution's profile.
}
\end{abstract}

\begin{keywords}
nonlocal diffusion operator,  integro-differential equations,  
quadrature method, gray-scott model, pulses.
\end{keywords}

\begin{AMS}
41A55 (approximate quadratures),
45J05 (integro-differential equation),
45P05 (integral operators),
65R20 (numerical methods for integral equations).
\end{AMS}

\section{Introduction}

In this paper, we use the quadrature scheme developed in \cite{jaramillo2021}, together with the method of lines,
to numerically study the effects of different boundary constraints on pulse solutions of the 1-d {\it nonlocal} Gray-Scott model. More precisely, we study the system
\begin{equation}\label{e:GS}
\begin{split}
u_t & = d_u Ku + A(1-u) - uv^2,\\
v_t & = d_v Kv  - Bv + uv^2,
\end{split}
\end{equation}
posed on a space domain $\Omega =[-L,L] \subset \R$ and a time interval $[0,T]$, with parameters $d_u, d_v, A, B, L>0$, and with $K$ defined via the integral operator
\begin{equation}\label{e:intoperator}
 Ku = \int_\R (u(y) - u(x)) \gamma(|x-y|) \;dy, \quad \gamma(z) \in L^1(\R) \cap L^2(\R). 
 \end{equation}
 
The type of boundary constraints we use are motivated by our choice of convolution kernel, $\gamma(z) $, which is assumed to be a radially symmetric and spatially extended function in $L^1(\R) \cap L^2(\R)$. As a result, in order to properly defined the operator $K$, we require information about the unknowns, $u$ and $v$, on all of $\R$.
 In other words, we need to know the value of these variables in the complement of our computational domain, $\Omega^c$.
 We use three different types of boundary constraints to determine the value of $u$ and $v$ in this region. A nonlocal Dirichlet boundary constraint, where the values of the unknown variables are prescribed in this outer domain.
 A homogeneous nonlocal Neumann boundary constraint, where the nonlocal flux is prescribed to be zero outside the computational domain. And a `free' boundary condition, where the decay rate of the unknown is prescribed but the value of the solution at the boundary is determined by the system. We use this last constraint as a means to approximate the solution when $\Omega = \R$.

Our motivation for studying the effects of boundary constraints comes from pattern forming systems 
that involve nonlocal diffusion, or long-range interactions, described by integral operators of the form given in \eqref{e:intoperator}. Examples include vegetation \cite{Eigentler2020, eigentler2018, bennett2019, baudena2013, pueyo2008, pueyo2010}, population \cite{othmer1988, hutson2003, lutscher2005, coville2007, li2018} and competition models \cite{sherratt2016, hetzer2012, kao2009random, hutson2003}, as well as models for the spread of infections and diseases \cite{medlock2003, Du_2023, kuniya2018, zhao2020}.
Because spatial discretizations of these operators using finite elements or quadrature methods lead to dense matrices,
most numerical schemes developed to study these systems use  instead Fourier spectral methods \cite{shima2004, avitabile2014, coombes2012, folias2005}.
However, this choice implicitly assumes that the equations satisfy periodic boundary conditions, which do not necessarily capture the boundary constraints imposed by the model or application.
Although this issue is recognized, a common justification for using a spectral approach is that most patterns of interest, like pulses, spot, and striped patterns, have a small length scale. It is then reasonable to expect that if the computational domain is large enough, boundary effects will be small. 
Thus, simulations are assumed to be valid as long as one restricts attention to regions well within the domain or, in the case of evolution problems, provided one stops the simulation before the boundary effects take place.

Nonlocal interactions such as those described by \eqref{e:intoperator} not only lead to computational challenges but also result in analytical difficulties.
For instance, one cannot use standard theory from spatial dynamics
or perturbation theory to directly study these problems. An alternative is to use convolution operators with Fourier symbols that are rational functions, since in this case one can precondition the equations with an appropriate differential operator and thus turn the integral equations into a set of partial differential equations. One is then able to apply techniques from dynamical systems to analyze these equations \cite{laing2006-2, Laing2014, coombes2005, coombes2003waves} and to also run numerical simulations using finite difference or spectral methods \cite{laing2006, avitabile2014, pinto2001a, pinto2001b}. The disadvantage is that one is then restricted to using convolution operators with 'nice' Fourier symbols.

Our goal in this paper is to determine if nonlocal boundary constraints lead to different numerical results when compared to spectral methods, and to see if the answer depends on the type of integral operator being considered, i.e. thin-tailed kernels (e.g. kernels with exponential decay) and fat-tailed kernels (e.g. kernels with algebraic decay). 
The paper is organized as follows. 
In Section~\ref{s:nonlocaldiffusion}, we describe the three nonlocal boundary constraints considered in this paper and the assumptions that guarantee the well-posedness of the 'free' boundary problem. 
The proposed numerical method, which combines a quadrature scheme together with an Adams-Bashforth time stepping, is introduced in Section~\ref{s:numericalmethods}.
In Section~\ref{s:continuation}, we study the generation of pulse solution for two kernels (thin- and fat-tailed) and compare the impact of Dirichlet, Neumann, free and periodic boundary constraints. In Section~\ref{s:continuation_cvg} we study the convergence of our proposed algorithm. A short discussion giving a heuristic argument for the existence of the pulse profiles found in the simulations is presented in Section~\ref{s:existence}.
Conclusions and future investigations are discussed in Section ~\ref{s:conclusion}.


\section{Nonlocal Diffusion and Boundary Constraints Models}\label{s:nonlocaldiffusion}

In this paper we look at the nonlocal Gray-Scott model \eqref{e:GS}, where we have replaced the Laplacian operator, commonly used in the description of this system,
with an integral operator modeling nonlocal diffusion.
In particular, we focus on maps, $K$, as given in equation \eqref{e:intoperator},
with convolution kernels, $\gamma(x,y) $, satisfying the following assumption:
\begin{Hypothesis}\label{h:kernel}
The convolution kernel $\gamma(x,y) = \gamma(|x-y|) = \gamma(z)$ is a positive radially symmetric
function living in $L^1(\R) \cap L^2(\R)$ with 
\begin{enumerate}[i)]
\item $\int_\R \gamma(z) \;dz = 1$,
\item  $\int_\R z^2 \gamma(z) \;dz < \infty$.
\end{enumerate}
\end{Hypothesis}

The above hypothesis encompasses many, if not all, convolution kernels used in the biological literature to model dispersal of animals and plant seeds, see \cite{bullock2017, pueyo2008}. In particular, this assumption covers both
thin-tailed kernels, which possess moment-generating functions, and fat-tailed kernels, which do not. 
The distinction between these two types of kernels is important from the mathematical point of view, since they can produce different effects. In particular, it is known that fat-tailed kernels lead to accelerating fronts in population dynamic models, whereas thin-tailed kernels lead to waves traveling with constant speed \cite{garnier2011, henderson2018, rodriguez2018, medlock2003}.
Here, we pick an example from each set and work with,
\begin{enumerate}
\item an exponential kernel, $ \gamma(z) = \frac{\sigma}{2} \exp(- \sigma |z|)$ with $\sigma \in \R$, and
\item a polynomial kernel, $ \gamma(z) = \dfrac{2 a^3}{\pi (z^2 + a^2)^2 } $ with $a \in \R$.
\end{enumerate}
These two functions not only provide examples of thin- and fat-tailed kernels, respectively, but are also commonly used in the ecological literature \cite{bullock2017}.

Notice that by rewriting the operator as
\[ Ku = \int_\R u(y) \gamma(x,y) \; dy - u(x)  \int_\R \gamma(x,y) \; dy,  \quad x\in \Omega =[-L,L],\]
one can then interpret the first term as the in-flux of individuals from $\R$ into location $x$, while the second term represents  the out-flux from location $x$ into $\R$, see \cite{du2012, andreu2010nonlocal} for more details. Therefore, the operator $K$  describes a nonlocal form of Fick's law, which then allows one to define zero flux boundary conditions by writing 
\[ K u = 0, \qquad x \in \Omega^c.\]
In this paper we use this last formulation to define nonlocal Neumann boundary constraints. 
We also consider nonlocal Dirichlet boundary constraints,
\[ u(x) = g(x), \qquad x \in \Omega^c,\]
where the function $g$ is given, and a form of 'free' boundary constraints, where the decay of the function at infinity is prescribed, i.e.
\[ u(x) = \frac{|L|^q}{|x|^q} u(\pm L), \qquad x \in \Omega^c,\]
for some $q>0$, and where the values of $u(\pm L)$ are determined by the algorithm. The reason for this last choice of boundary condition comes from our interest in approximating 
the whole real line problem, i.e. $\Omega = \R$.
In the next subsection, we provide the necessary background material to justify this formulation. 
More details on how we numerically implement these constraints is given in Section \ref{s:numericalmethods}.

\subsection{Properties of the Nonlocal Operator $K$}

In this subsection we make the following more restrictive assumption on the convolution operator $K$ defined in \eqref{e:intoperator}.

\begin{Hypothesis}\label{h:kernel_analytic}
Let $K$ be as in \eqref{e:intoperator} and consider $\eta_0>0$. We assume that the Fourier symbol, $\hat{K}(\xi)$ is analytic on a strip of the complex plane $\{ \xi \in \C \mid | \Im(\xi) | < \eta_0\}$. In addition, we suppose $\hat{K}(\xi)$ is radially symmetric and has one simple zero at the origin with corresponding Taylor expansion,
\[ \hat{K}(\xi) = - \alpha |\xi|^2 + \rmO(|\xi|^4) \qquad \mbox{as} \quad |\xi| \sim 0, \quad \alpha >0.\]

\end{Hypothesis}

Notice that any exponentially decaying kernel $\gamma(z)$ satisfying Hypothesis \ref{h:kernel} generates an operator $K$ satisfying the above assumption.

Our goal is to use the above hypothesis to prove a Fredholm alternative for the operator $K$ when the physical domain $\Omega$ is assumed to be $\R$. This is done by picking suitable weighted Sobolev spaces to define the domain and range of the operator. These spaces in turn provide us with decay properties for the unknown, which we can then use to approximate the solution for large values of $x$.

First recall that an operator is Fredholm if its range is closed and if it has a finite dimensional kernel (nullspace) and co-kernel. What is relevant for us is that while $K: L^2(\R) \longrightarrow L^2(\R)$ does not have a closed range, due to having a zero eigenvalue embedded on its essential spectrum, we can recover Fredholm properties for this operator  if its domain is given instead by spaces that impose some level of algebraic growth or decay at infinity. 
More precisely, we need to use Kondratiev spaces, $M^{s,p}_\gamma(\R)$, and weighted Sobolev spaces, $W^{s,p}_\gamma(\R)$, which we define next.

\begin{Definition}[Kondratiev Spaces]
Let $s$ be a non-negative integer, $p \in (1,\infty)$ and $\gamma \in \R$. We then let $M^{s,p}_\gamma(\R)$ denote the space of locally summable, $s$ times weakly differentiable real-valued functions endowed with the norm,
\[ \| u\|_{M^{s,p}_\gamma(\R)} = \sum_{j=0}^s \| \partial^j_x u\|_{L^p_{j+\gamma}(\R)} \quad \mbox{where} 
\quad 
\| u\|_{L^p_\gamma(\R)} = \| (1+|x|^2)^{\gamma/2} u\|_{L^p(\R)}.
\]
\end{Definition}

\begin{Definition}[Weighed Sobolev Spaces]
Let $s$ be a non-negative integer, $p \in (1,\infty)$ and $\gamma \in \R$. We then let $W^{s,p}_\gamma(\R)$ denote the space of locally summable, $s$ times weakly differentiable real-valued functions endowed with the norm,
\[ \| u\|_{W^{s,p}_\gamma(\R)} = \sum_{j=0}^s \| \partial^j_x u\|_{L^p_{\gamma}(\R)} \quad \mbox{where} 
\quad 
\| u\|_{L^p_\gamma(\R)} = \| (1+|x|^2)^{\gamma/2} u\|_{L^p(\R)}.
\]
\end{Definition}

Notice that the above definitions can be extended to non-integer values $s$ by interpolation and
to negative values of $s$ by duality. 
For values of $p \in (1,\infty)$ the above spaces are also reflexive, so that $(M^{s,p}_\gamma(\R))^* = M^{-s,q}_{-\gamma}(\R)$ and  $(W^{s,p}_\gamma(\R))^* = W^{-s,q}_{-\gamma}(\R)$ with $q$ being the conjugate exponent of $p$. The pairing between a function $f \in L^p_\gamma(\R)$ and an element in the dual space $g \in L^q_{-\gamma}(\R)$ is then given by 
\[ \langle f, g \rangle = \int_\R f g \;dx.\]

Analyzing the above spaces we also find that functions in $M^{s,p}_\gamma(\R)$ decay algebraically if $\gamma>0$, whereas they may grow algebraically if $\gamma<0$. 
In addition, while  Kondratiev spaces gain a degree of algebraic localization with each derivative, the above weighted Sobolev spaces do not.  In particular, we have the following Lemma.
\begin{Lemma}\label{l:decay}
Given $\gamma>0$, a function $f \in M^{1,p}_\gamma(\R)$ satisfies $|f(x)| \leq \| f'\|_{L^p_{\gamma+1}(\R)}|x|^{1/q- (\gamma+1)}$ as $|x| \to \infty$.
\end{Lemma}
\begin{proof}
Because $\gamma >0$ we may write $|f(x)| < \int_\infty^x | f'(y) (1+ y^2)^{(\gamma+1)/2} | \cdot |(1+ y^2)^{-(\gamma+1)/2} | \;dy.$ The result then follows by applying H{\"o}lder's inequality.
\end{proof}

The above spaces, together with Hypothesis \ref{h:kernel}, then lead to Theorem \ref{t:fredholm} given below. This result is proved rigorously in \cite{jaramillo2019}, where more general operators defined over $L^2(\R,Y)$, with $Y$ a separable Hilbert space which commutes with the actions of translations, are considered. Here we re-state the results of \cite{jaramillo2019} for the case when $Y=\R$.

\begin{Theorem}\label{t:fredholm}
Let $p \in (1,\infty)$ with $q$ its conjugate exponent, and let $\gamma \in \R$ be such that
$ \gamma + 2 + 1/p \notin \{ 1,2 \}$. Suppose as well that the convolution operator $K:M^{2,p}_\gamma(\R) \longrightarrow W^{\ell,p}_{\gamma+2}(\R)$ satisfies Hypothesis \ref{h:kernel_analytic}. Then, with appropriate value of the integer $\ell$, this operator is Fredholm and
\begin{itemize}
\item for $\gamma< 1- 2 - 1/p$ it is surjective with kernel (null-space) spanned by $\mathbb{P}_2$;
\item for $\gamma> -1+ 1/q$ it is injective with co-kernel spanned by $\mathbb{P}_2$;
\item for $j-1-2 +1/q < \gamma< j + 1-2 - 1/p$, where $j \in \N, 1 \leq j <2$, its kernel is spanned by $\mathbb{P}_{2-j}$ and its co-kernel is spanned by $\mathbb{P}_j$.
\end{itemize}
Here $\mathbb{P}_j$ is the $j-$dimensional space of all polynomials of degree less or equal to $j$.
\end{Theorem}

The above theorem leads to the following corollary, which shows conditions under which the integral operator $K$ satisfies a Fredholm alternative. In addition, given the equation $Ku =f$, the corollary also lets us infer decay rates at infinity for the solution $u$ based on similar knowledge of the source term $f$. This last result then allows us to approximate the whole real line problem ($\Omega = \R$) by assuming a-priori the rate of decay of $u$, and reformulating this information as a  'free' boundary constraints.

\begin{Corollary}\label{t:decay}
Let $s \in \Z \cup [2, \infty)$ and $p \in (1,\infty)$ with $q$ its conjugate exponent. Consider the convolution
operator $K$ defined by Hypothesis \ref{h:kernel_analytic} 
Then, letting
\[\begin{array}{c c c}
K: M^{s,p}_\gamma(\R) & \longrightarrow & W^{s,p}_{\gamma+2}(\R)\\
u & \longmapsto & K u,
\end{array}\]
with $\gamma > -1 + 1/q$, the equation $K u = f$ has a unique solution with $|u(x)|< C|x|^{1-1/p - (\gamma+1)}$ for large $|x|$, provided the right hand side $f \in  W^{s,p}_{\gamma+2}(\R)$ satisfies,
\[ \langle f, 1 \rangle =0 \qquad \langle f, x \rangle =0.\]

\end{Corollary}

{\bf Remark:} Notice that because integral operators generated using fat-tail kernels do not have analytic Fourier symbols, they do not satisfy Hypothesis \ref{h:kernel_analytic}. Therefore, in principle,  the results presented in this subsection do not apply to these maps. Nonetheless, as shown in Section \ref{s:continuation}, the numerical scheme using the 'free' boundary condition does converge for both types of kernels.


\section{Numerical Methods}\label{s:numericalmethods}
In this section, the numerical scheme used to approximate solutions of the nonlocal Gray-Scott model \eqref{e:GS} is described in more detail. 
First, we apply the method of lines to rewrite our problem as a system of ODEs and present the algorithm used to perform the time evolution of these equations.  We then summarize the quadrature method developed in \cite{jaramillo2021}, which we use here to evaluate the operator $K$, and present
the resulting numerical schemes for the three different boundary constraints considered in this paper. Finally, we show that the proposed methods converge using the method of manufactured solutions.

\subsection{Method of Lines}
By first discretizing the problem in space, we can view the resulting equations as a system of ODEs.
Here we use a second-order Adams-Bashforth scheme to do the time stepping. Denoting by $dt>0$ a given time step and by $g^n$ the approximation of a function $g$ at time $t_n= n dt$, our method leads to the following time marching algorithm.

Given $u^n, v^n, u^{n-1}, v^{n-1}$, find $u^{n+1}$ and $v^{n+1}$ solution of
\begin{align*}
u^{n+1}\mathbbm{1}_\Omega = u^n \mathbbm{1}_\Omega  & + \frac{3}{2} dt \left[ d_u \; K u^n + u^n(v^n)^2 \mathbbm{1}_\Omega - f(1-u^n) \mathbbm{1}_\Omega \right]\\
& - \frac{1}{2} dt \left[ d_u \; K u^{n-1} + u^{n-1}(v^{n-1})^2\mathbbm{1}_\Omega - f(1-u^{n-1}) \mathbbm{1}_\Omega\right],
\end{align*}
and
\begin{align*}
v^{n+1}\mathbbm{1}_\Omega = v^n \mathbbm{1}_\Omega  & + \frac{3}{2} dt \left[ d_v \; K v^n - u^n(v^n)^2 \mathbbm{1}_\Omega - (f+\kappa) v^n \mathbbm{1}_\Omega \right]\\
& - \frac{1}{2} dt \left[ d_v \; K v^{n-1} - u^{n-1}(v^{n-1})^2\mathbbm{1}_\Omega - (f+\kappa) v^{n-1} \mathbbm{1}_\Omega\right],
\end{align*}
where $\Omega = [-L,L]$ and $\mathbbm{1}_\Omega$ denotes the indicator function. Notice that the notation emphasizes the fact that while the equations are posed on the domain $\Omega$, the operator $K$ uses information about the unknowns on all of $\R$. We also note that the above algorithm is fully explicit, which leads to a decoupled formulation in the sense that $u^{n+1}$ and $v^{n+1}$ can be computed in sequential order.

In the next subsection we give a brief overview of how the spatial discretization of the operator $K$ changes depending on the boundary constraints being used. We will see that the implementation of nonlocal Dirichlet and `free' boundary constraints is very similar, since in both cases we have explicit information about the unknowns, $u$ and $v$, in $\Omega^c$. This is not the case for the Neumann problem, where we assume that $K u =0$ and $K v =0$ on $\Omega^c$. In this case, as when considering $\Omega = \R$, we do not have an explicit description of the unknowns on $\Omega^c$.  As in \cite{jaramillo2021}, we get around this issue by imposing $K u =0$ and $K v=0 $ only on a bounded outer domain $\Omega_o = [-2L,-L] \cup [L,2L]$, and by then approximating $u$ and $v$ in $\R \setminus \Omega \cup \Omega_o$ using `free' boundary conditions.

\subsection{Quadrature Method} We describe the spatial discretization of the operator $K$ for three types of boundary constraints (Dirichlet, Neumann, 'free') using a quadrature method. Notice first that the operator $- K$ has the equivalent formulation,
\[ - Ku = \int_\R (u(x) - u(x-y) ) \gamma(|y|)\;dy, \quad x \in \Omega. \]
Introducing the nodes $x_i = i h$ with $i \in \Z \cap[\-L,L]$ and the mesh size $h>0$, we may then write
\[ -[Ku ](x_i) = \int_\R (u(x_i) - u(x_i-y) ) \gamma(y) \;dy, \qquad x_i \in \Omega.  \]
To discretize the operator,  we split the integral  into four terms, 
\begin{align*}
 -[K \ast u](x_i) =  & \underbrace{ \int_{-h}^h (u(x_i) - u(x_i-y) ) \gamma(y) \;dy}_{I}  \\[2ex]
 + & \underbrace{ \int_{h\leq |y|\leq L_W} (u(x_i) - u(x_i-y) ) \gamma(y) \;dy}_{II}  \\[2ex]
  + & \underbrace{ u(x_i) \int_{|y|>L_W} \gamma(y) \;dy }_{III} \quad - \quad \underbrace{ \int_{|y|>L_W} u(x_i - y) \gamma(y) \;dy}_{IV},
 \end{align*}
where $L_W= 2L$.

As shown below, Integral $I$ can be calculated in the same way for all three types of boundary constraints, while the value of integral $III$ can always be determined explicitly. The discretization scheme  for the remaining integrals, Integral $II$ and $IV$, depends on the type of boundary constraints used. Notice in particular that Integral $IV$ only samples the unknown $u$  outside the domain of interest, $\Omega$. Indeed,  a short calculation shows that for all $|y|>L_W$ the points $(x_i-y)$ are in $ \Omega^c = \R \setminus \Omega$. 
We now describe how Integrals $I,II$ and $IV$ are handled.

{\bf Integral $I$:}
To discretize the Integral $I$, we follow the approach taken in \cite{huang2014} where the same quadrature method is used to approximate the fractional Laplacian operator, $(-\Delta)^{\alpha/2}$ with $0<\alpha< 2$. In this case, the kernel $\gamma(x,y) \sim \frac{1}{|x-y|^{1+ \alpha}}$, is singular near the origin, so that the value of the integral near this point needs to be interpreted as a Principal Value.  This leads to
\begin{align*}
I = & \int_{-h}^h (u(x_i) - u(x_i-y) ) \gamma(y) \;dy \\
 = & \lim_{\eps \to 0} \int_{\eps}^h (u(x_i) - u(x_i-y) ) \gamma(y) \;dy + \int_{-h}^{-\eps} (u(x_i) - u(x_i-y) ) \gamma(y) \;dy\\
  = & \int_0^h  [ 2u(x_i) - u(x_i + y) - u(x_i -y) ] \gamma(y) \;dy\\
  = & \int_0^h  [ u''(x_i) y^2 + \rmO(y^4)]  \gamma(y) \;dy.
 \end{align*}
where the last line comes from Taylor expanding each element in the integrand about the point $x_i$. Using central difference and integrating, we then arrive at
 \begin{align}\label{e:I}
 \begin{split}
 I& = -[ u(x_i +h) -2u(x_i) +u(x_i-h)] \frac{1}{h^2} f_1(h) + \rmO(h^2)\\
  & = \Big [ u(x_i) - u(x_{i-1}) \Big ] f_1(h) + \Big[ u(x_i) - u(x_{i-1}) \Big] f_1(h) + \rmO(h^2),
  \end{split}
  \end{align}
 with $$f_1(h) = \int_0^h y^2 \gamma(y) \;dy.$$
 
While we do not consider singular kernels in this manuscript, we retain this formulation for convenience.

 {\bf Integral II:} To discretize the second integral, we first approximate its integrand using polynomial interpolation. To this end, we let $T_h$ denote the standard tent function given by
\begin{equation}\label{e:hat}
 T_h(t) = \left \{ 
\begin{array}{c c}
1- |t|/h & \mbox{for } |t|\leq h,\\
0 & \mbox{otherwise}.
\end{array}
\right.
\end{equation}
Then, for any smooth function $f: \R \longrightarrow \R$,
we let $Pf$ represent the linear interpolant of $f$. That is,
\[  Pf(y) = \sum_{j \in \Z} f(x_j) T_h(y - x_j).\]
With this formulation, and letting $f(y) = (u(x_i) - u(x_i-y)) $, we have that
\begin{align*}
 II = &\int_{h< |y|< L_W} f(y) \;dy \\
 = & \int_{h< |y|< L_W} [P f(y)] \gamma (y) \;dy + \rmO(h^2)\\
 = & \int_{h< |y|< L_W} \sum_{j \in \Z} f(x_j) T(y - x_j)\gamma(y) \;dy + \rmO(h^2)\\
= &  \sum_{j \in \Z} f(x_j) \int_{h< |y|< L_W}T(y - x_j) \gamma(y)\;dy + \rmO(h^2).
\end{align*}
Thus, we get
\begin{equation}\label{e:II}
II =  \sum_{j \in \Z } (u(x_i) - u(x_i - x_j)) \underbrace{\int_{h< |y| <L_W} T(y-x_j) \gamma(y)\;dy}_{w_j} + \rmO(h^2)
\end{equation}
We will see later in this subsection that for the kernels considered in this paper, the weights $w_j$ depends on the type of boundary constraints considered and can be calculated explicitly, see equation \eqref{eq:def_wj}.

{\bf Integral $IV$:}
In what follows we  briefly describe the three types of boundary constraints  being considered, and address how the integral $IV$ is approximated in each case.
 
{\it Nonlocal Dirichlet:} In the case of nonlocal Dirichlet boundary constraints, 
we assume that
 \[ u(x) = g(x), \quad \mbox{ for all} \quad |x|>L,\]
for some function $g(x)$ that is given. We can then calculate Integral $IV$ explicitly, either numerically or if possible analytically,  via
\begin{equation}\label{e:IV_D}
 IV =  \int_{|y|>L_W} g(x_i - y) \gamma(y) \;dy .
 \end{equation}

{\it Free Boundary:} 
For the whole real line problem, i.e. when $\Omega = \R$, the unknown $u$ is approximated on $\Omega^c$ by an algebraically
decaying function, where the best guess for the decay can be derived based on the problem under consideration. That is, we let
\begin{equation}\label{e:neum}
u(x_i) = \frac{u(\pm L)}{g( \pm L)} g(x_i), \quad \mbox{for all} \quad |x_i|>L,
\end{equation}
where in this case $g$ is a given function sharing the same decay, or growth rate, as $u$.
For instance, we may impose $u(x) \sim g(x) = \frac{1}{|x|^q}$ for some $q \in \R$.  As mentioned in Section \ref{s:nonlocaldiffusion}, this assumption follows from Theorem \ref{t:fredholm}, its Corollary \ref{t:decay}, and the hypothesis placed on the operator $K$. However, notice that while we impose the decay/growth conditions on $u$ in $\Omega^c$, we do not prescribe the value of $u$ at the end points $x = \pm L$. As a result, Integral $IV$ can be calculated using
\begin{equation}\label{e:IV_RL}
IV = \frac{u(\pm L)}{g(\pm L)} \int_{|y|>L_W} g(x_i - y) \gamma(y) \;dy.
\end{equation}

\begin{Remark}
    
In Section \ref{s:continuation} we run simulations of the nonlocal Gray-Scott model using an algebraically decaying convolution kernel $\gamma(x,y)$ to define the operator $K$, see \eqref{e:intoperator}. 
Despite this kernel not satisfying the hypothesis of Theorem \ref{t:fredholm}, we are able to numerically show that the method converges.
\end{Remark}

{\it Nonlocal Neumann:}
As in \cite{jaramillo2021} we treat the Neumann problem as a special case of the whole real line problem, i.e. when $\Omega = \R$.
More precisely, if the physical domain is given by $\tilde{\Omega} = [-\ell, \ell]$, we let $L = 2\ell$ and then define the computational domain of $K$
as $\Omega = [-L,L]$. This means that while, for example, the equation \eqref{e:GS} is satisfied for $x_i \in \tilde{\Omega} = [-\ell, \ell]$, the value of $Ku(x_i)$  is approximated for all points $x_i \in  [-L,L]$. In particular, we have that $Ku(x_i) =0$ for $x_i \in \tilde{\Omega}_o =[-2\ell, -\ell] \cup [\ell, 2\ell]$. Then to discretize $Ku$  in $\Omega$, as in the case of the whole real line problem, we assume that $u$ has the same decay properties on $\Omega^c$ as some function $g(x) \sim 1/|x|^q$, for some $q\in \R$, and  calculate the integral $IV$
 using equation \eqref{e:IV_RL}.

{\bf Resulting Numerical Scheme.} The discretizations given in equations \eqref{e:I}, \eqref{e:II}, and \eqref{e:IV_D} or \eqref{e:IV_RL}, yield the following numerical scheme to approximate the operator $-K$
\[ -[Ku](x_i) \sim \sum_{-M}^M (u(x_i) - u(x_i - x_j) ] w_j + C u_i + D_i,\]
where $ M$ is an even number satisfying $L = \frac{M}{2}h$, for some mesh size $h$. Here, the terms $C$ and $D_i$ are given by
\begin{align*}
C =  &\int_{|y|>L_W} \gamma(y) \;dy, \\[2ex]
 D_i  = &  \left\{ 
\begin{array}{c l}
  \int_{|y|>L_W} g(x_i - y) \gamma(y) \;dy &\quad \mbox{Dirichlet Case,}\\[4ex]
  \dfrac{u(\pm L)}{g(\pm L)} \int_{|y|>L_W} g(x_i - y) \gamma(y) \;dy  & \quad \mbox{Real Line and Neumann Case},
  \end{array}
  \right.
 \end{align*}
while the weights $w_j$ satisfy
 \[ w_j = \left \{
 \begin{array}{c c c}
 0 & \mbox{for} & j = 0,\\[2ex]
  f_1(h) + \int_{|y|\geq h} T_h(y-x_j) \gamma(y) \;dy & \mbox{for} & j =  \pm 1,\\[2ex]
  \int_{|y|\geq h} T_h(y-x_j) \gamma(y) \;dy & \mbox{for} & 1< |j| \leq M.
 \end{array}
 \right.\]
Notice that when $j =0$ we have $u(x_i) = u(x_i - x_j)$, and thus we can define the weights $w_0$ arbitrarily. 
On the other hand, the weights $w_{\pm 1}$ are determined by integrating functions whose support intersects the interval $[-h,h]$. They  are therefore the sum of integrals $I$ and $II$, the former being represented here by the term $f_1(h)$.
Because the kernel $\gamma(y)$ is symmetric, we can conclude that $w_j = w_{-j}$, and arrive at the following table
\begin{equation}\label{eq:def_wj}
w_j = \left \{
 \begin{array}{c c c}
 0 & \mbox{for} & j = 0,\\[2ex]
  f_1(h) - F'(x_1) +\dfrac{1}{h} [F(x_2) - F(x_1)] & \mbox{for} & |j| =   1,\\[2ex]
 \dfrac{1}{h}[ F(x_{j+1} - 2F(x_j) + F(x_{j-1})] & \mbox{for} & 1< |j| < M,\\[2ex]
  F'(x_M) +\dfrac{1}{h} [F(x_{M-1}) - F(x_M)] & \mbox{for} & |j| =   M.\\
 \end{array}
 \right.
\end{equation}
with $F$ satisfying $F''=\gamma$ and
where we use the Lemma \ref{l:weights}, proved in  \cite{huang2014} and shown below for completeness, to calculate explicitly the weights $w_j$.
\begin{Lemma}\label{l:weights}
Let $F(t)$ be a $C^2([-2h,2h])$ function and let $T_h(t)$ be as in \eqref{e:hat}. Then 
\[ \int_{-h}^h T_h(t) F''(t) \;dt = \frac{1}{h} [ F(h) - 2F(0) + F(-h) ],\]
\[ \int_{0}^h T_h(t) F''(t) \;dt = - F'(0) + \frac{1}{h} [  F(h) - F(0) ].\]
\end{Lemma}

We note that the above quadrature method yields an approximation of order two in space, i.e. $h^2$, of the operator $-K$. The following section describes the fully discretized algorithm for the nonlocal Gray-Scott model.

\subsection{Fully discretized algorithms}
Combining the quadrature method with the method of lines, introduced in the previous subsections leads to a system of ordinary differential equations which we solve using an explicit second-order Adams-Bashforth method. 

\begin{algorithm}
\caption{Nonlocal Dirichlet and `Free' Boundary Constraints \label{algorithm1}}
\begin{algorithmic}[1]

\STATE \textbf{Preliminary:} Number of nodes $M$,  discretized operator $K$, domain $\Omega$, 
parameters $p$, initial condition $W_0=[u_0,v_0]$, $t_0=0$, vector field $G= G(t,W,K,p)$.

\STATE \textbf{Input:}  Tolerance $tol$, time step $dt$, Max number of time steps $nmax$.
\STATE \textbf{Trial Step:} $t_1 = t_0 + dt$, \quad $W_1 = W_0 + dt*G(t_0,W_0,K)$.
\FOR{$n= 0$ \TO $nmax$}
\STATE $t_{n+1} = t_n +dt$
\STATE $W_{n+1} = W_n + \frac{3}{2} dt * G(t_n,W_n,K) - \frac{1}{2} dt * G(t_{n-1},W_{n-1},K) $
\STATE error = max $|W_{n+1}-W_n|$
\IF{error  $< tol$} 
\STATE Break
\ENDIF
\ENDFOR
\end{algorithmic}
\label{al:dirichlet}
\end{algorithm}

In the case of nonlocal Dirichlet and 'free' boundary constraints, our approach leads to the scheme presented in Algorithm \ref{algorithm1}. To be consistent with the numerical illustrations presented in Section \ref{s:continuation}, we present a version of the scheme that identifies stationary states. That is, when the difference between our approximation at two consecutive times, i.e. $W_{n+1}$ and $W_n$, becomes smaller than a given positive tolerance, the algorithm stops. The algorithm can be straightforwardly modified to approximate the solutions of the nonlocal Gray-Scott model on a given time interval $[0,T]$ by setting $nmax=T/dt$ and $tol=-1$ in the above algorithm. This strategy will be employed in the following section to verify that the resulting algorithm converges with order two.

In the case of nonlocal Neumann boundary constraints, where at each time step we need to determine the values of the unknowns in the computational domain of $K$, $\Omega = [-L,L]$, we split our algorithm into two parts.

First, given $u_n$ and $v_n$ in $\Omega$, we compute $Ku_n$ and $Kv_n$  using the quadrature method with 'free' boundary conditions to approximate the value of these functions on $\Omega^c$. This allows us to evolve the nonlocal Gray-Scott equations on the physical domain $\tilde{\Omega}= [-L/2,L/2]$, where the equations are defined, and obtain  the values of $u_{n+1},v_{n+1}$ in this region. 
Next, we determine the value of $u_{n+1}$ and $v_{n+1}$ on the outer domain $\tilde{\Omega}_o = [-L, -L/2] \cup [L/2, L]$, using the boundary constraint. That is, we consider the {\it extension} problem
\[
\begin{array}{c c}
Ku = & 0\\ 
Kv = &0 
\end{array} \qquad x \in \tilde{\Omega}_o = [-L, -L/2] \cup [L/2, L], \]
\[
\begin{array}{c c}
u = & u_n\\ 
v = &v_n 
\end{array} \qquad x \in \tilde{\Omega} =[-L/2,L/2]. \]

Generally, the above equations can be solved as follows. Letting  $u_1 = u\mid_{[-L,-L/2]}$, $u_2 = u\mid_{\tilde{\Omega}}$,  and $u_3 = u \mid_{[L/2,L]}$, we may write
\[ Ku = \begin{bmatrix} 
K_{11} & K_{12} & K_{13}\\
K_{21} & K_{22} & K_{23}\\
K_{31} & K_{32} & K_{33}
\end{bmatrix} 
\begin{bmatrix}
u_1\\ u_2 \\ u_3
\end{bmatrix} = 
\begin{bmatrix} 
0 \\ f \\ 0
\end{bmatrix},\]
where $K_{ij}$ represent the block matrices that make up our  discretized operator $K$. The vector $f$ is at this point an arbitrary vector, since as shown below it does not enter our calculations.
Indeed, given that we start with information about the value of $u$ inside the domain $\tilde{\Omega}$, the vector $u_2$ is assumed to be known. We may thus re-arrange the above expression as
\[ \underbrace{\begin{bmatrix} 
K_{11} &  K_{13}\\
K_{31}  & K_{33}
\end{bmatrix} }_{\tilde{K}}
\begin{bmatrix}
u_1 \\ u_3
\end{bmatrix} = 
\begin{bmatrix} 
- K_{12} u_2\\ -K_{32} u_2,
\end{bmatrix}\]
which can now be solved for $[ u_1, u_3] = u\mid_{\tilde{\Omega}_o}$.
The above scheme is summarized in Algorithm \ref{extend}, where we use the notation presented here.

\begin{algorithm}
\caption{Extension \label{extend}}
\begin{algorithmic}[1]

\STATE \textbf{Preliminary:} Inner domain $\Omega =[-L/2,L/2]$, outer domain $\Omega_o=[-L,-L.2] \cup [L/2,L]$.
\STATE \textbf{Input:}  Operator $K$, function $u: \Omega \longrightarrow \R$.
\STATE $\tilde{K}, K_{12},K_{32} \gets$ Restrict $[K]$
\STATE $f_u \gets [-K_{12}u_i, -K_{32} u_i] $
\STATE $[u_{o,1},u_{o,2}] \gets \tilde{K}^{-1} f_u$
\STATE $ u_e \gets [ u_{o,1}, u, u_{o,2}]$
\RETURN {$u_e$}
\end{algorithmic}
\end{algorithm}

\begin{algorithm}
\caption{Nonlocal Neumann Boundary Constraints \label{algorithm2}}
\begin{algorithmic}[1]
\STATE \textbf{Preliminary:} Number of nodes $M$,  discretized operator $K$, domain $\Omega$, 
parameters $p$, initial condition $W_0=[u_0,v_0]$, $t_0=0$, vector field $G= G(t,W,K,p)$.
\STATE \textbf{Input:}  Tolerance $tol$, time step $dt$, Max number of time steps $nmax$.
\STATE \textbf{Trial Step:} $t_1 \gets t_0 + dt$, \quad $[u_1,v_1] = W_1 \gets W_0 + dt \cdot G(t_0,W_0,K,p)$.
\STATE $u_{e} \gets$ Extension $(K, u_1)$ \quad  $v_{e} \gets$ Extension $(K, v_1)$ 
\STATE $W_{e,1} \gets [ u_e, v_e]$
\FOR{$n= 0$ \TO $nmax$}
\STATE $t_{n+1} = t_n +dt$
\STATE $[u_{n+1},v_{n+1}]=W_{n+1} \gets W_n + \frac{3}{2} dt * G(t_n,W_{e,n},K,p) - \frac{1}{2} dt * G(t_{n-1},W_{e,n-1},K,p) $
\STATE $u_{e} \gets$ Extension $(K, u_{n+1})$ \quad  $v_{e} \gets$ Extension $(K, v_{n+1})$ 
\STATE $ W_{e,n+1} \gets [u_e,v_e]$
\STATE error = max $|W_{n+1}-W_n|$
\IF{error$\leq tol$} 
\STATE Break
\ENDIF
\ENDFOR
\end{algorithmic}
\label{al:neumann}
\end{algorithm}

\subsection{Convergence Results} In this section we apply the method of manufactured solutions to verify the convergence of the proposed algorithm. Recall first that the proposed quadrature method is of order two, see \cite{jaramillo2021} for a complete proof, which means that the spatial error is of order $\rmO(h^2)$, with $h$ the mesh size. Second, we use a second-order Adams-Bashforth method, so the time error is order $\rmO(dt^2)$, with $dt$ the time step.

Since the spatial convergence of the quadrature method presented here
 has already been studied extensively in \cite{jaramillo2021}, without loss of generality,
 we can focus our numerically investigations on a problem that uses nonlocal Dirichlet boundary constraints. Unlike \cite{jaramillo2021}, where stationary problems of the form $Ku=f$ were studied, we now consider the nonlocal Gray-Scott model \eqref{e:GS}, which we supplement with additional source terms denoted by $f_u$ and $f_v$. The modified system reads:
\begin{equation*}
\begin{split}
u_t & = d_u Ku + A(1-u) - uv^2 + f_u,\\
v_t & = d_v Kv  - Bv + uv^2 + f_v.
\end{split}
\end{equation*}
To use the method of manufactured solutions, we consider the following kernel $\gamma$ and solution pair $(u,v)$:
\begin{equation} \label{eq:cvg_test_dirichlet_BC}
\left\{ 
    \begin{array}{c}
    \gamma(y) = 0.5 \exp(-|y|),\\
    u(x,t) = 0.5 \cos(t)  
    \left(1.0 + \sin(\pi(x-0.5)) \right) (1-x^2) \exp(1.0-x^2),\\
    v(x,t) = \cos(t^2) \cos(\pi x /2) x^3 \sin(\pi x).
    \end{array}
\right.
\end{equation}
The functions $u(x)$ and $v(x)$ presented above are then solutions to the nonlocal Gray-Scott model 
provided we chose appropriate source terms, $f_u$ and $f_v$. To run the test, the computational domain is set to $\Omega = [-1,1]$ and the time interval to $[0,1]$. Homogeneous Dirichlet boundary constraints are enforced on both $u$ and $v$. The system parameters are set as follows:
 $$d_u = 0.05, d_v = 0.01, A = 6, \text{ and } B = 8.$$
We perform a series of six simulations with $M\in\{40,80,160,320,640,1280\}$ quadrature nodes. 
The mesh size $h$ is defined as $h=|\Omega|/M=2/M$, and the time step $dt$ is set to $2h=4/M$. Our results are displayed in Table \ref{tab:cvg_test_dirichlet_BC} which shows the evolution of the $L^2$ errors in $u$ and $v$ when we refine the mesh size and time step. The results are consistent with a second-order algorithm. Indeed, we recover $L^2$ errors that behave in $\rmO(h^2 + dt^2)$ for both unknowns $u$ and $v$. 

We note that similar convergence studies can be performed for the Real Line and Neumann problems. We refer to Section \ref{s:continuation_cvg} for a convergence study of the proposed methods to stationary pulse solutions of the nonlocal Gray-Scott model. We also refer to \cite{jaramillo2021} for convergence tests performed on stationary problems of the form $Ku = f$ with 'free' and 'Neumann' boundary constraints.

\begin{table} 
\centering
\begin{tabular}{  m{1.7cm} m{1.8cm} m{1cm} m{1.7cm} m{1.2cm} m{1.7cm} m{1.2cm}}
\hline \hline
 dt & h & M & Error $u$ & Order $u$ & Error $v$ & Order $v$\\
\hline \hline
      0.05 & 0.025   & 40 & 8.35E-5  & -  & 4.50E-5   &  -   \\
      0.025 & 0.0125   & 80 & 2.16E-5  & 1.95  & 1.05E-5   & 2.10    \\
      0.0125 & 0.0625   & 160 & 5.47E-6  & 1.98  & 2.52E-6   & 2.06    \\
      0.00625 & 0.003125  & 320 & 1.38E-6  & 1.99  & 6.13E-7   & 2.04    \\
      0.003125 & 0.0015625  & 640 & 3.46E-7  & 2.00  & 1.51E-7   & 2.02    \\
      0.0015625 & 0.00078125   & 1280 & 8.66E-8  & 2.00  & 3.75E-8   & 2.01     \\ 
\hline 
\end{tabular}
\caption{Problem \eqref{eq:cvg_test_dirichlet_BC}: $L^2$ errors for Algorithm \ref{al:dirichlet} using nonlocal Dirichlet boundary constraints. The time step is denoted by $dt$, the mesh size by $h$ and the number of quadrature nodes by $M$.}
\label{tab:cvg_test_dirichlet_BC}
\end{table}

\section{Pulse Solutions and the Effect of Boundary Conditions}
\label{s:continuation}

In this section, we simulate the nonlocal Gray-Scott model using the three boundary constraints described in Section \ref{s:numericalmethods}, i.e. nonlocal Dirichlet, nonlocal Neumann, and `free' boundary constraints. Our goal is to understand how these boundary constraints affect the pulse pattern profiles in this system. We compare these results with numerical simulations obtained using a Fourier spectral method that assumes periodic boundary conditions. Finally, we consider the effects of expanding the computational domain.

Since our numerical experiments focus on pulse solutions, we choose parameters and initial conditions that converge to this type of patterns. That is, we pick
\[d_u = 1, \quad d_v= 0.01, \quad  A = 0.01, \quad B = (0.01)^{1/3}/2,\]
and set the initial conditions as follows
\begin{equation}\label{eq:continuation_IC}
u_0(x) = 1 - 0.67 \dfrac{1}{\alpha \sqrt{2\pi}} 
    \exp(-0.5 (\frac{x}{\alpha})^2),
\qquad \text{ and } \qquad 
v_0(x) = 0.925 \dfrac{\beta}{ \alpha 2\sqrt{2} \Gamma(\beta^{-1})}
    \exp(- (\frac{|x|}{\alpha})^\beta),
\end{equation}
where $\beta=3$ and $\alpha=0.1$.

For the nonlocal Dirichlet and the `free' boundary constraint case we use a physical/computational domain $\Omega =[-75/4,75/4]$, while in the nonlocal Neumann case we pick $\tilde{\Omega} = [-75/4 -75/4]$ and  use a computational domain $\Omega =[-75/2, 75/2]$ to calculate the nonlocal operator $K$.

As mentioned in Section \ref{s:nonlocaldiffusion}, we consider the following convolution kernels, 
\begin{enumerate}
\item  $ \gamma(z;\sigma) = \frac{\sigma}{2} \exp(- \sigma |z|)$ with $\sigma \in \R$, and
\item  $ \gamma(z;a) = \dfrac{2 a^3}{\pi (z^2 + a^2)^2 } $ with $a \in \R$,
\end{enumerate}
which then define our nonlocal operator via formula \eqref{e:intoperator}.
Notice that in the case of the exponential kernel, as the parameter $\sigma $ increases the spread of the operator decreases, see Figure \ref{fig:kernels}. Consequently, the operator behaves more and more like a Laplacian as the value of $\sigma$ grows, but regains its nonlocal character as the value of $\sigma$ decreases.
On the other hand, in the case of the algebraic kernel we find that as the value of the parameter $a$ increases, the spread of the operator increases. So we have the opposite effect, the smaller the parameter $a$ is, the more the operator $K$ behaves like the Laplacian. When running the simulations, we also vary the parameters $\sigma$ and $a,$ and track how these changes affect the profiles of pulse solutions. 

\begin{figure}[ht] 
   \centering
   \includegraphics[width=2.8in]{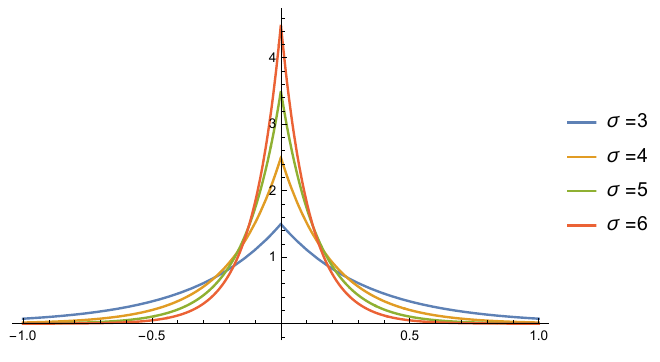} 
   \includegraphics[width=2.8in]{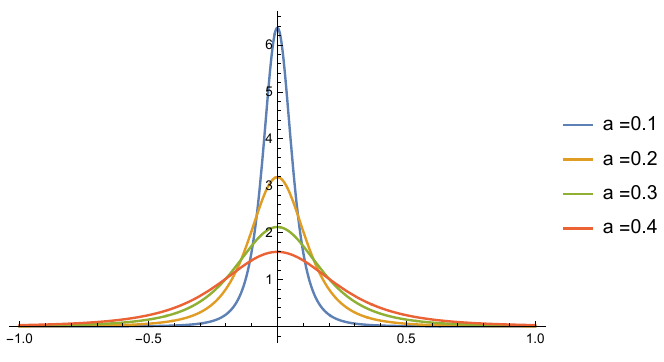}  
   \caption{Convolutions kernels. Left: exponential kernel  $\gamma(z) = \frac{\sigma}{2} \exp(- \sigma |z|)$. Right: algebraic kernel  $ \gamma(z) = \dfrac{2 a^3}{\pi (z^2 + a^2)^2 } $.}
   \label{fig:kernels}
\end{figure}

Figures \ref{fig:Exponential_Nonlocal_BC} and \ref{fig:Algebraic_Nonlocal_BC} depict pulse solutions obtained using the exponential and algebraic kernels, respectively, to define the operator $K$. Both figures also account for nonlocal Dirichlet, Neumann and 'free' boundary constraints. Overall we find that when the convolution operator $K$ is close to the Laplacian, i.e. for small values of $a$ or large values of $\sigma$, boundary effects are minimal and the resulting pulse solutions are qualitatively similar in all cases. That is, the solutions have a smooth pulse profile, although nonlocal Dirichlet boundary constraints lead to shorter pulses. 
On the other hand, when the spread of the nonlocal operator is large (i.e. small $\sigma$ values or large $a$ values), we observe that the exponential kernel leads to pulses with a \emph{mesa-profile}, while the algebraic kernel leads to a \emph{cat-ear} profile. This behavior is more prominent when using nonlocal Dirichlet boundary constraints, showing that when the kernels are wide both, the type of nonlocal boundary condition and the type of convolution kernel used (i.e. thin- or fat-tailed),  affect the profile of pulse solutions.
\begin{figure}[ht] 
   \centering
 \includegraphics[width=2.75in]{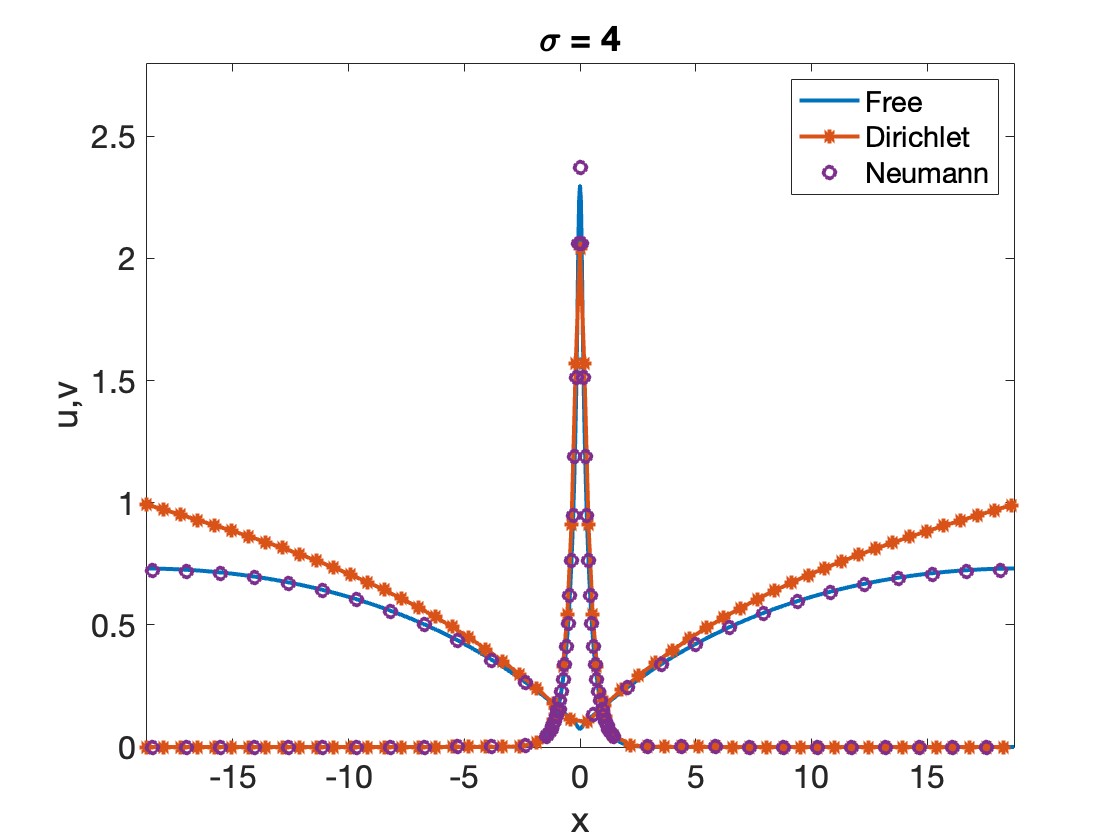}
 \includegraphics[width=2.75in]{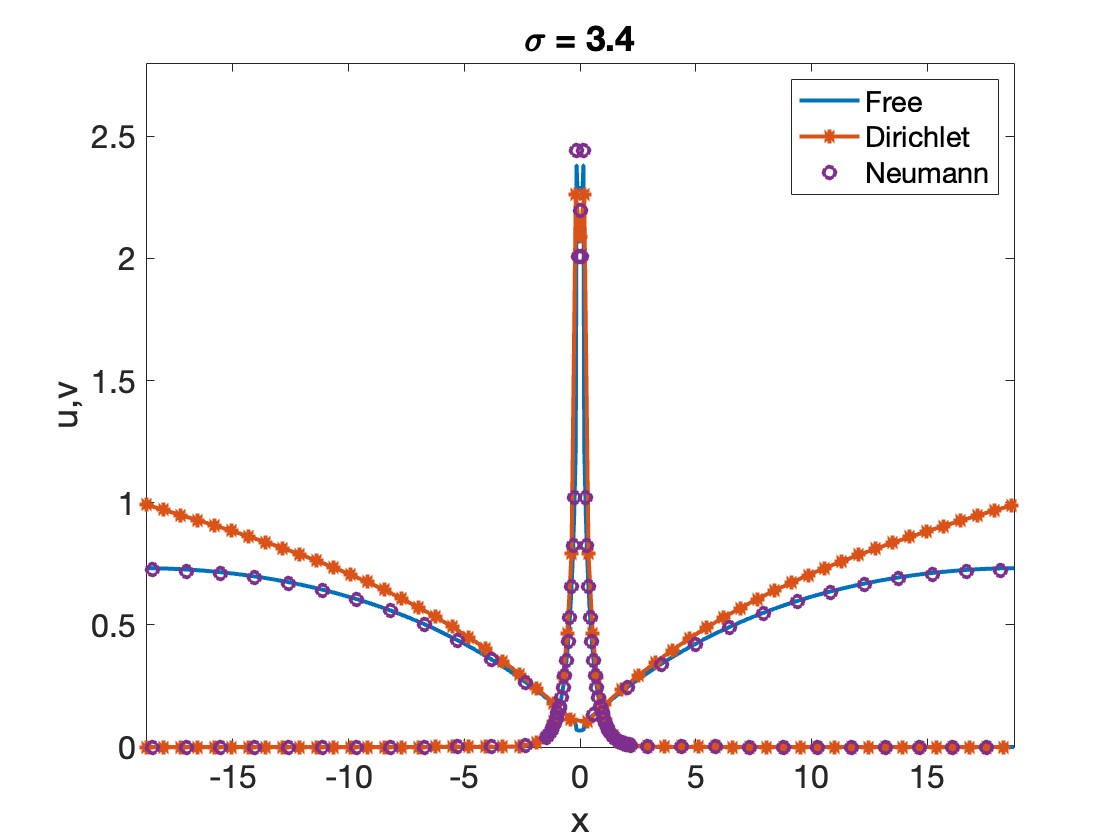}
 
\includegraphics[width=2.75in]{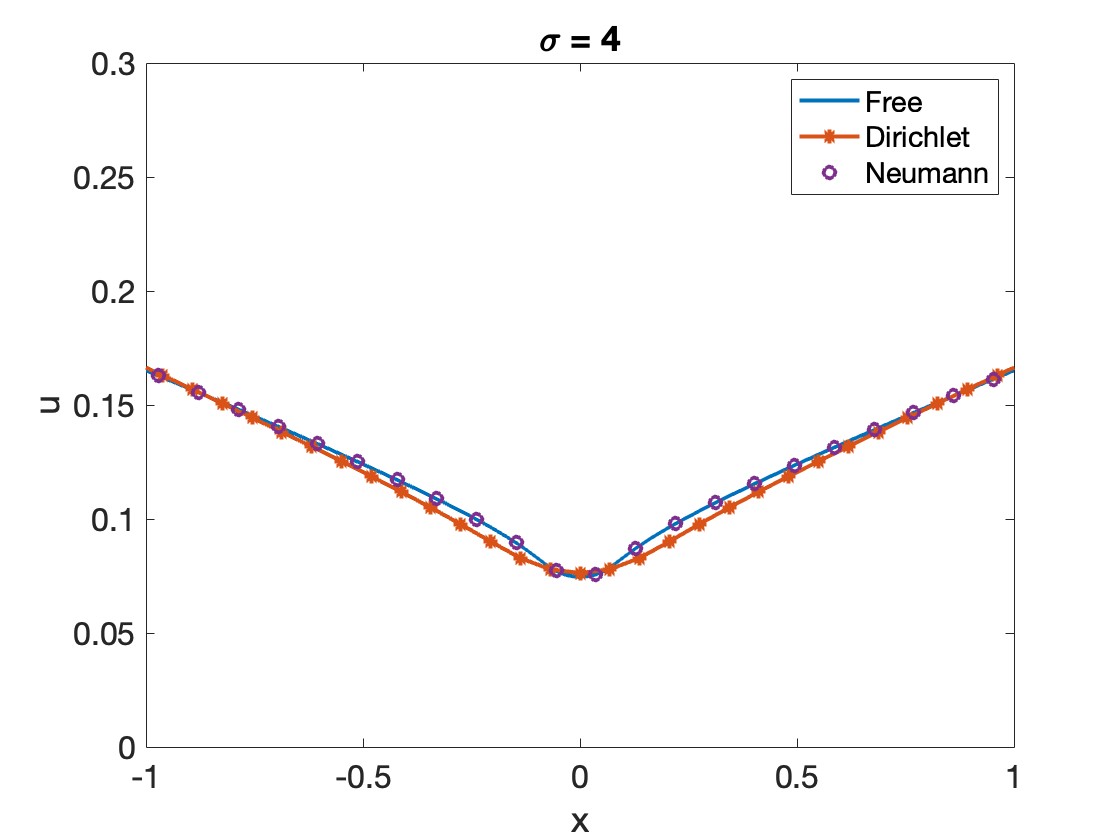}
 \includegraphics[width=2.75in]{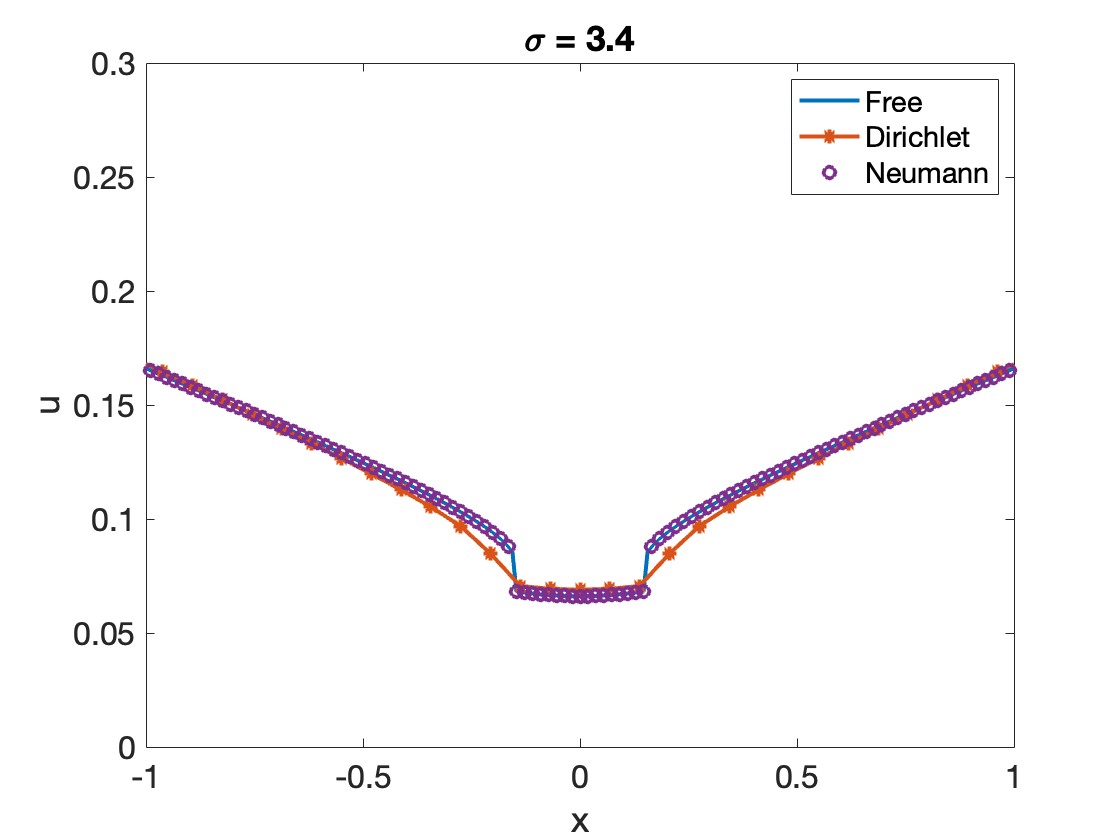}

\includegraphics[width=2.75in]{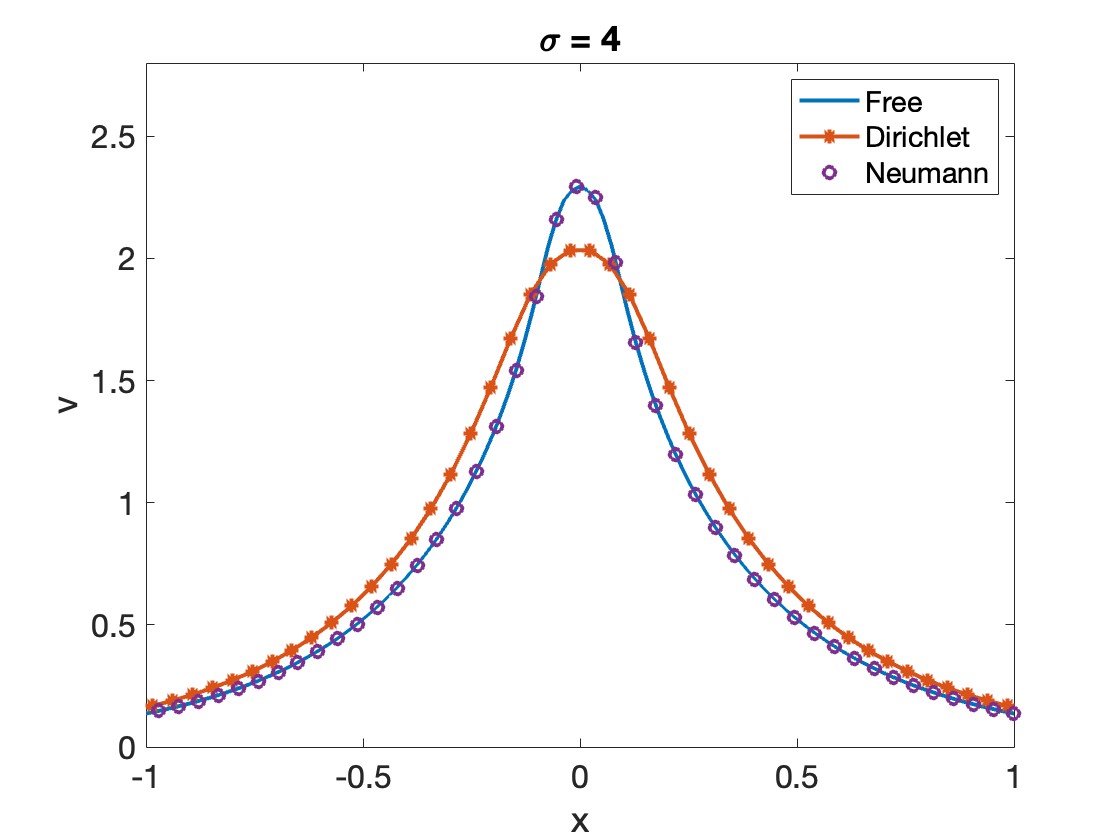}  
  \includegraphics[width=2.75in]{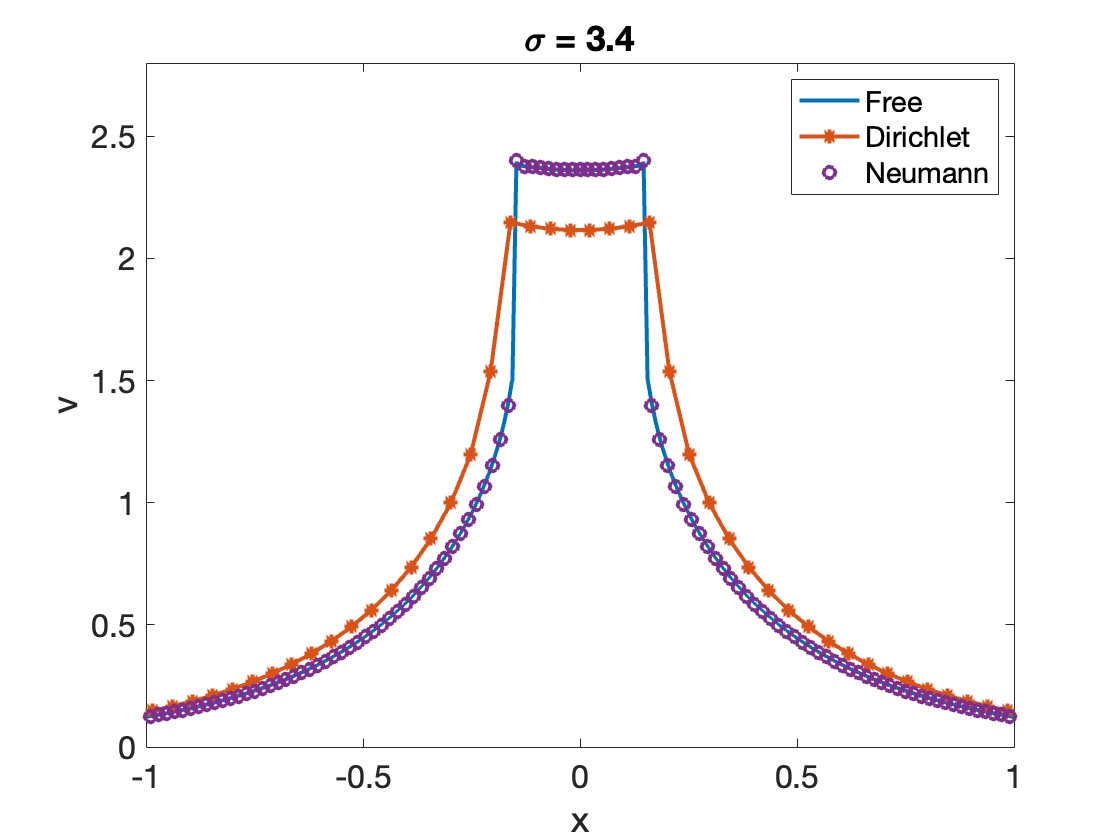}
\caption{Pulse solutions of nonlocal Gray-Scott model \eqref{e:GS} with map $K$ defined by exponential kernel with $\sigma =4$ (left column), and $\sigma =3.4$ (right column). Simulations compare all three nonlocal boundary constraints and use $M = 2^{13}$ nodes.}
\label{fig:Exponential_Nonlocal_BC}
\end{figure}
\begin{figure}[ht] 
   \centering
 \includegraphics[width=2.75in]{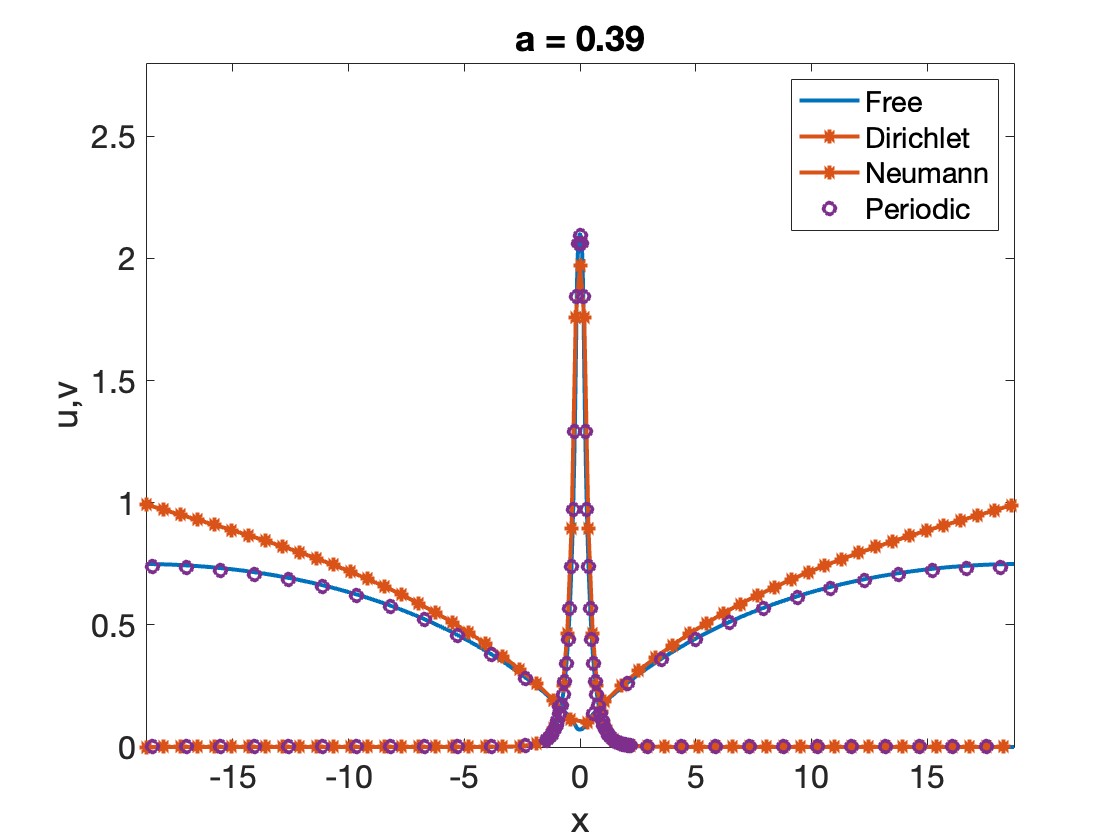}
  \includegraphics[width=2.75in]{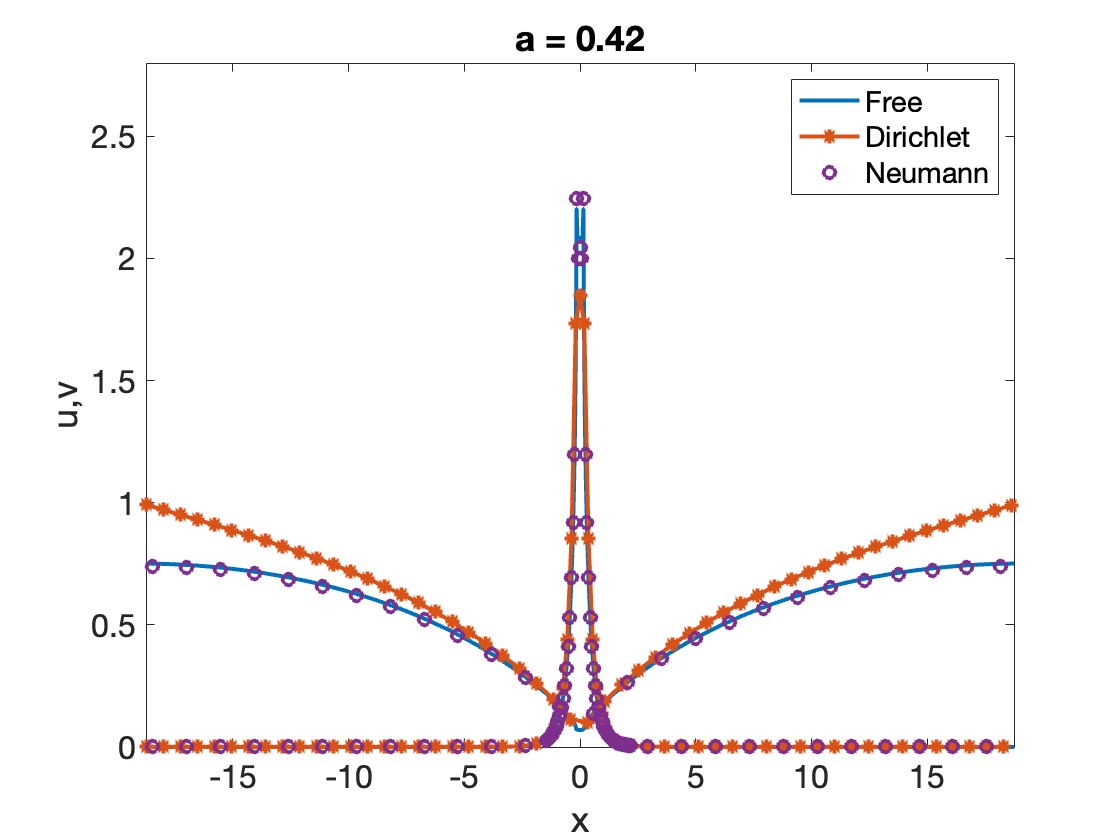}
 
 \includegraphics[width=2.75in]{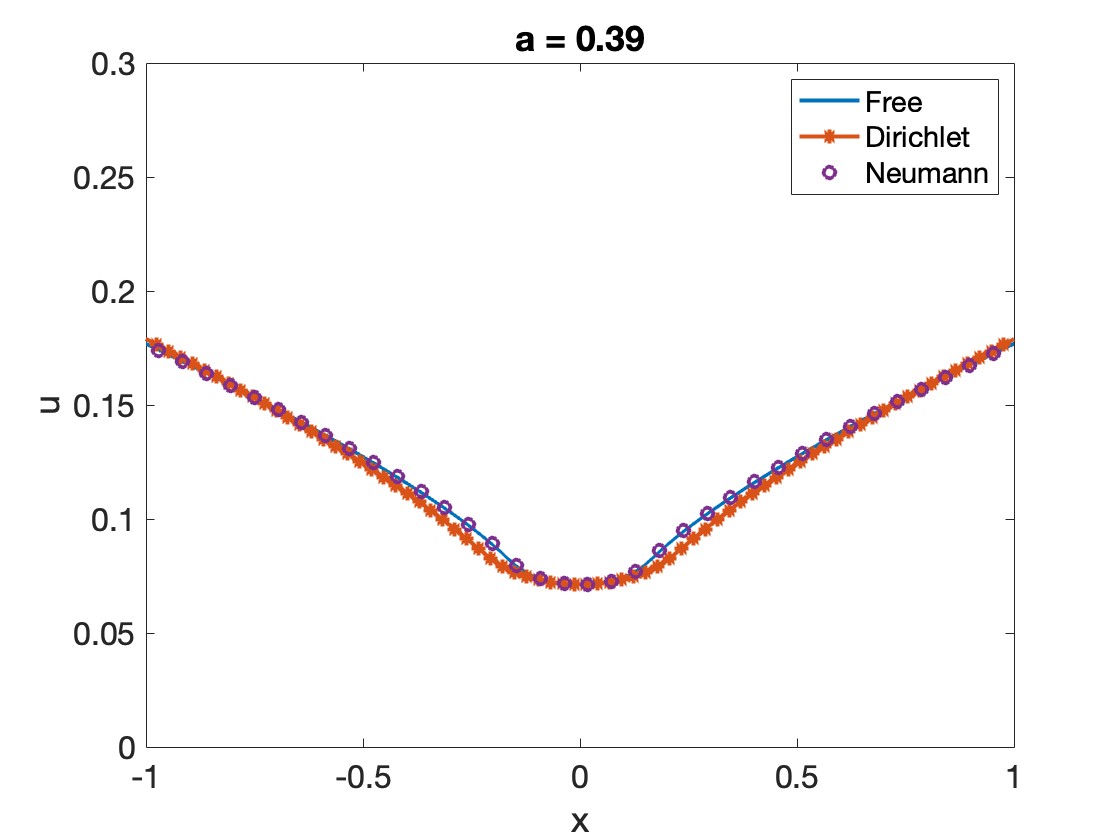}
 \includegraphics[width=2.75in]{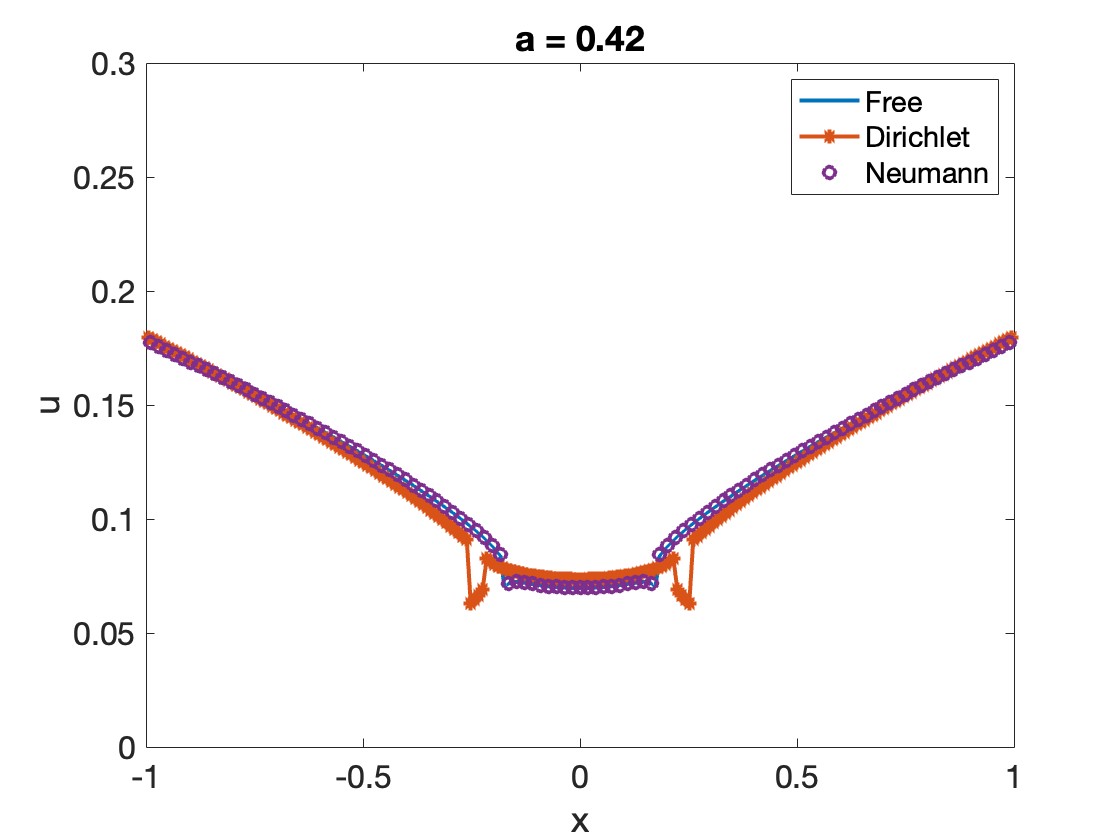}
 
 \includegraphics[width=2.75in]{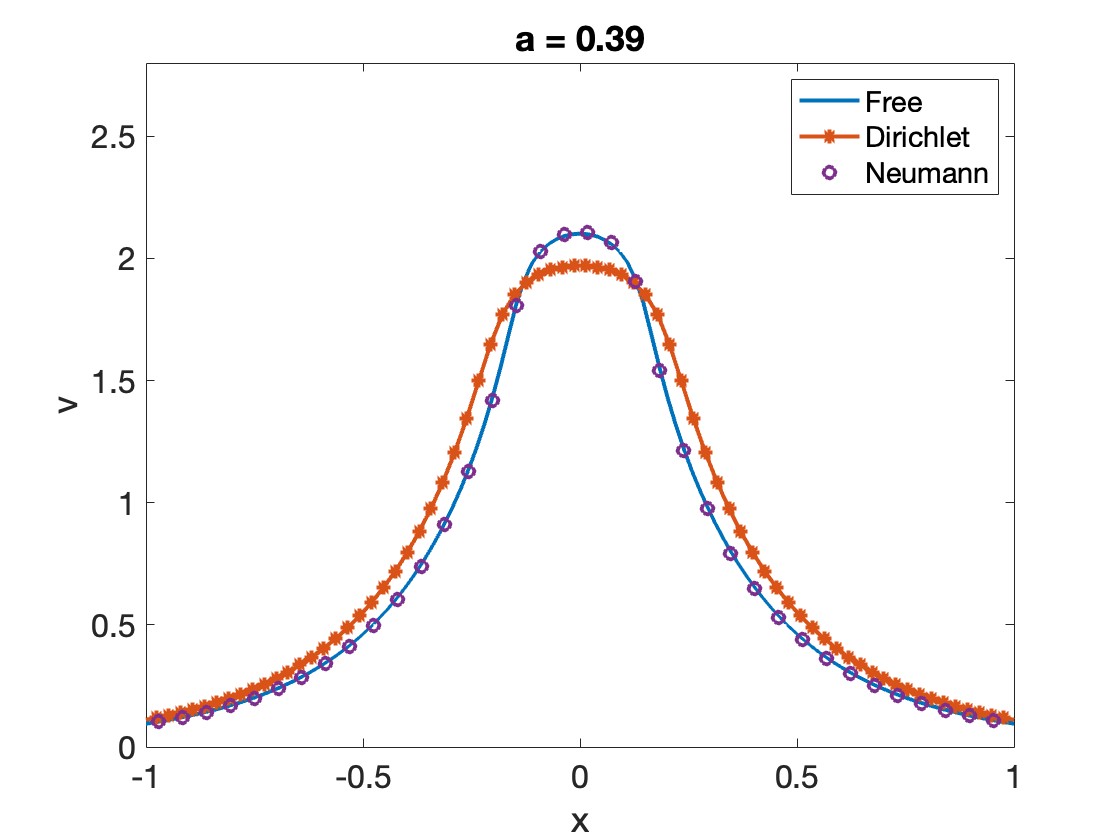}
 \includegraphics[width=2.75in]{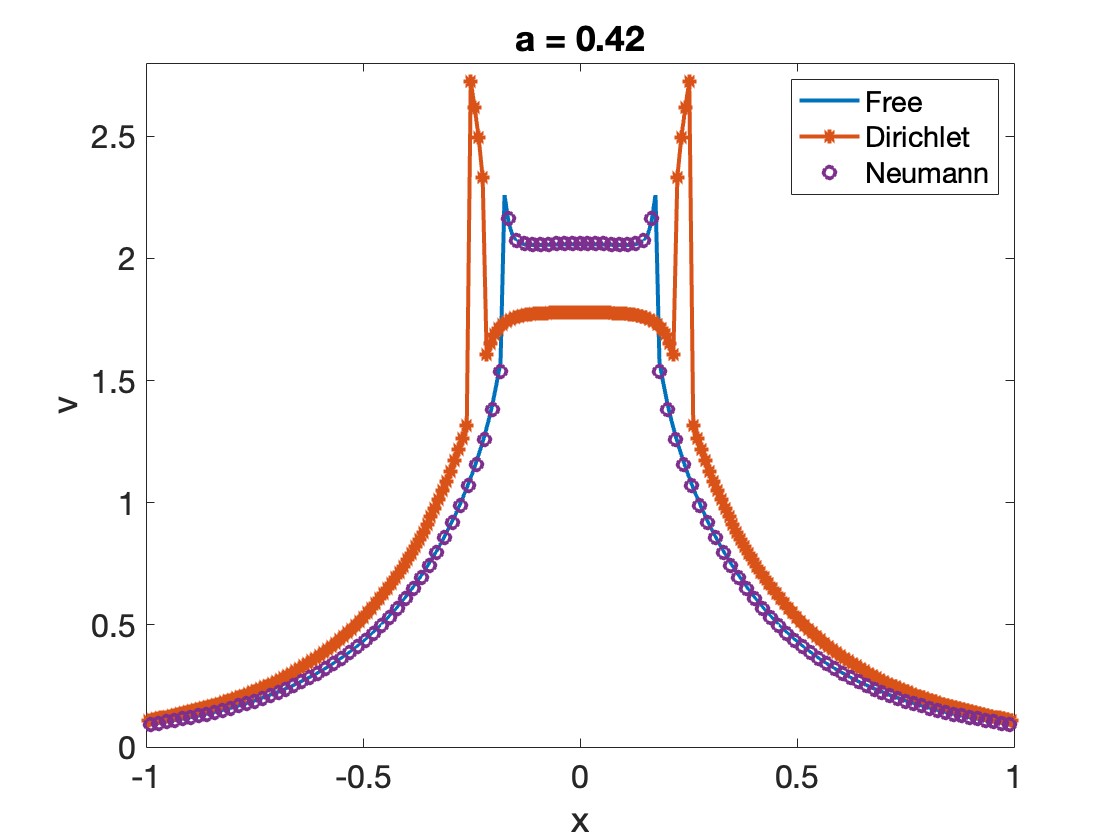}
\caption{Pulse solutions of nonlocal Gray-Scott model \eqref{e:GS} with map $K$ defined by algebraic kernel with $a =0.39$ (left column), and $a =0.42$ (right column). Simulations compare all three nonlocal boundary constraints and use $M = 2^{13}$ nodes.}
\label{fig:Algebraic_Nonlocal_BC}
\end{figure}

Figures \ref{fig:Exponential_RL_Periodic} and \ref{fig:Algebraic_RL_Periodic} compare the effect of nonlocal versus local boundary constraints by plotting pulse solutions obtained using the 'free' boundary constraints together with solutions obtained using periodic boundary conditions. In the latter case, we used the Fast Fourier Transform (FFT) to transform our convolution operator into a multiplication operator. The resulting system of ordinary differential equations is then solved using a second-order Backward Euler scheme. These figures show again that when the nonlocal operators $K$ behave like the Laplacian, boundary effects are minimal, while for wide kernels different boundary conditions lead to different results. Indeed, we find that periodic boundary conditions produce pulses with tops that oscillate, while nonlocal boundary constraints result in solutions with sharp corners, i.e. \emph{mesa-profiles} when using the exponential kernel or \emph{cat-profiles} when using the algebraic kernel. 

\begin{figure}[ht] 
   \centering
 \includegraphics[width=2.75in]{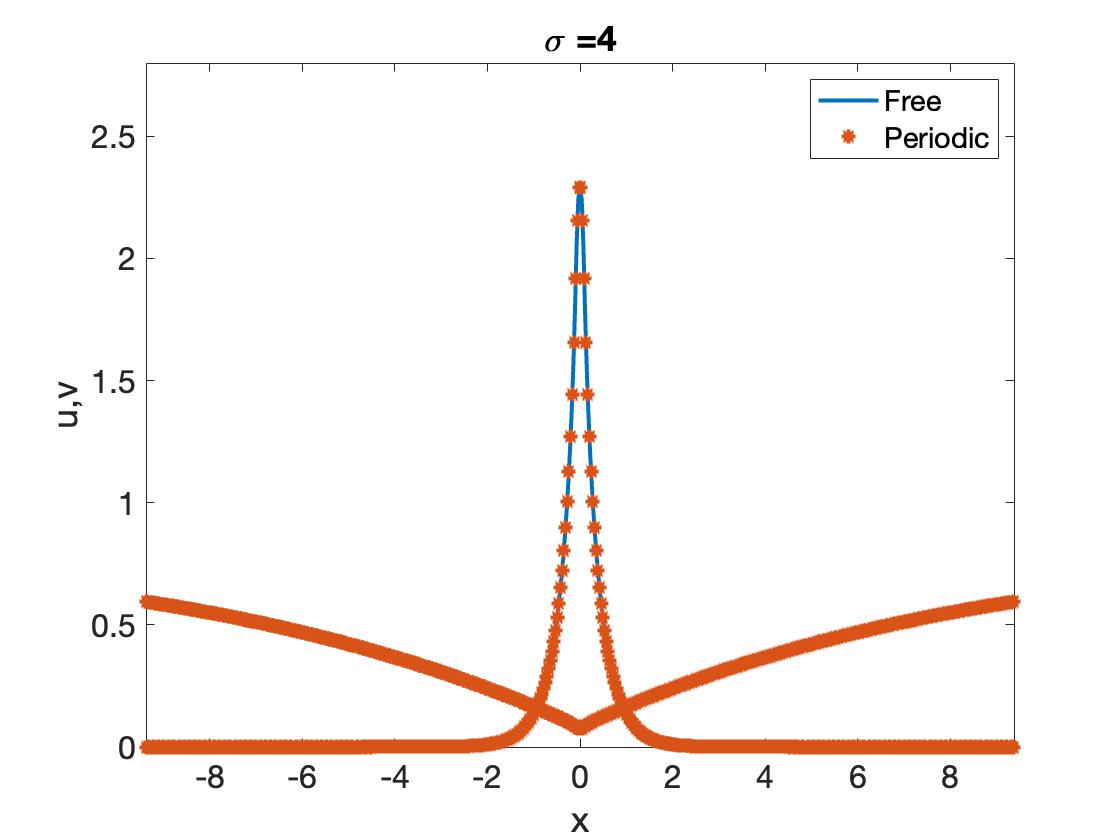}
  \includegraphics[width=2.75in]{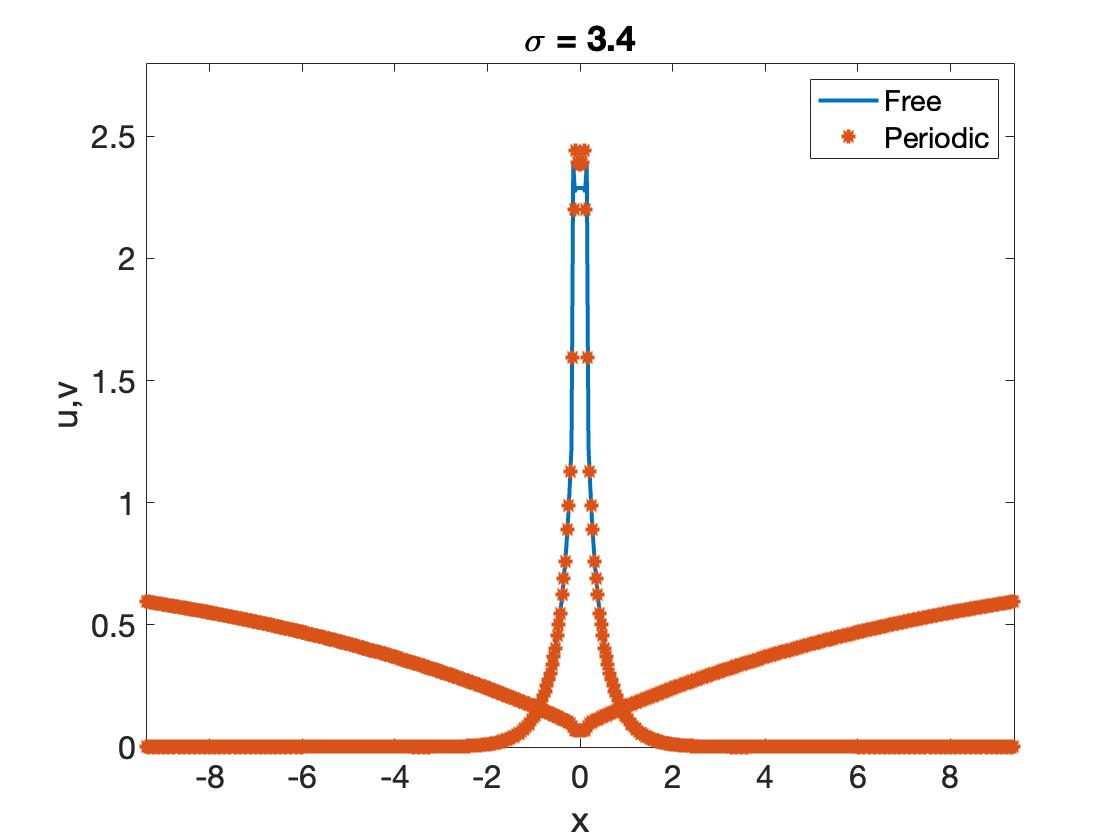}
 
\includegraphics[width=2.75in]{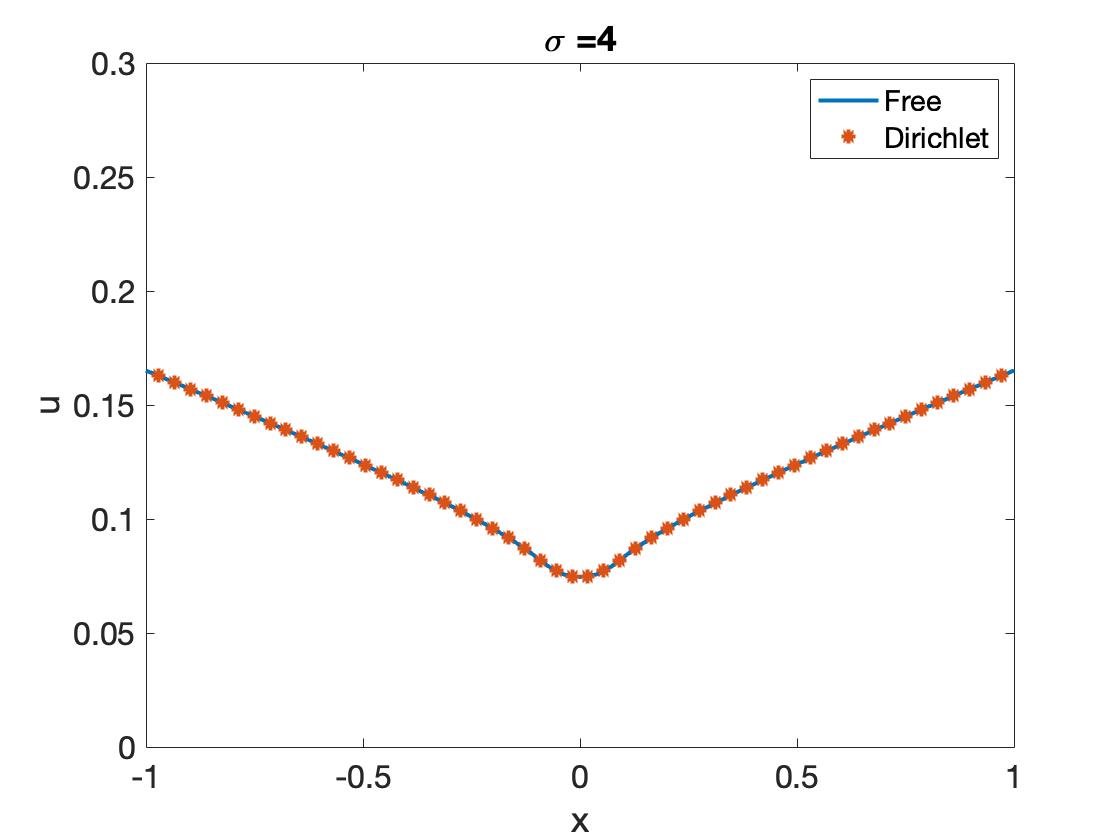}
 \includegraphics[width=2.75in]{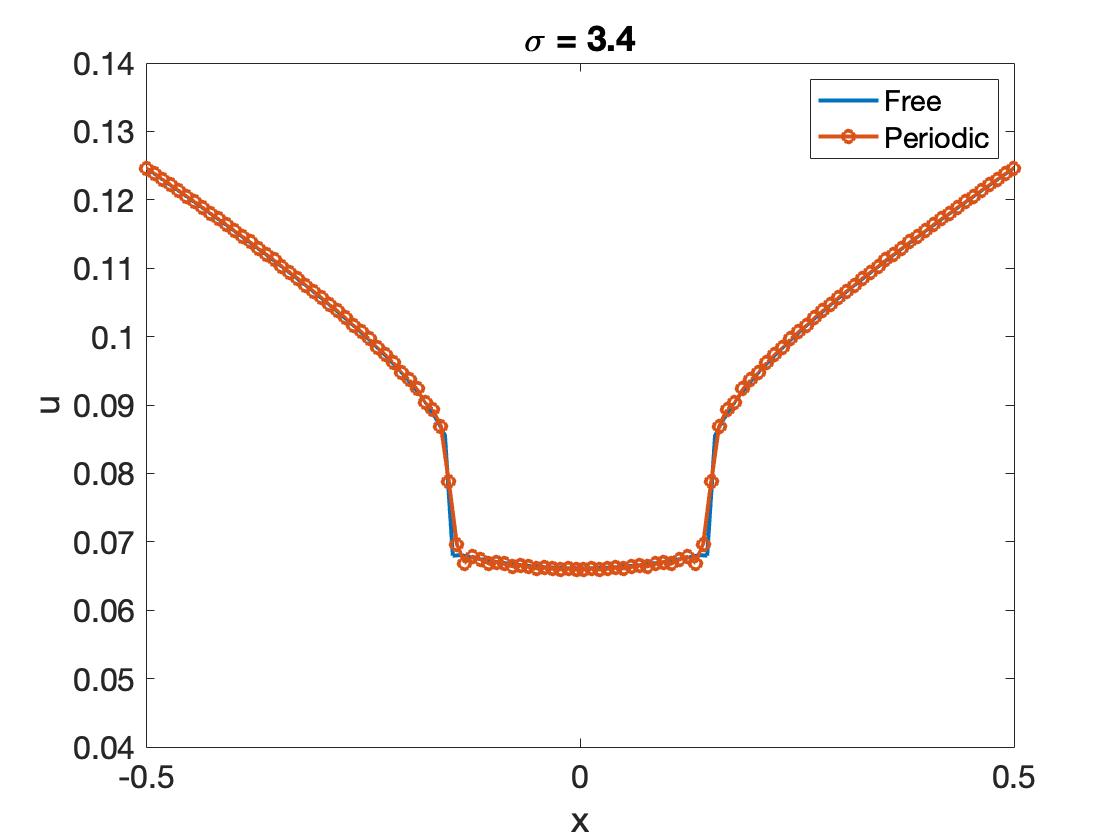}

\includegraphics[width=2.75in]{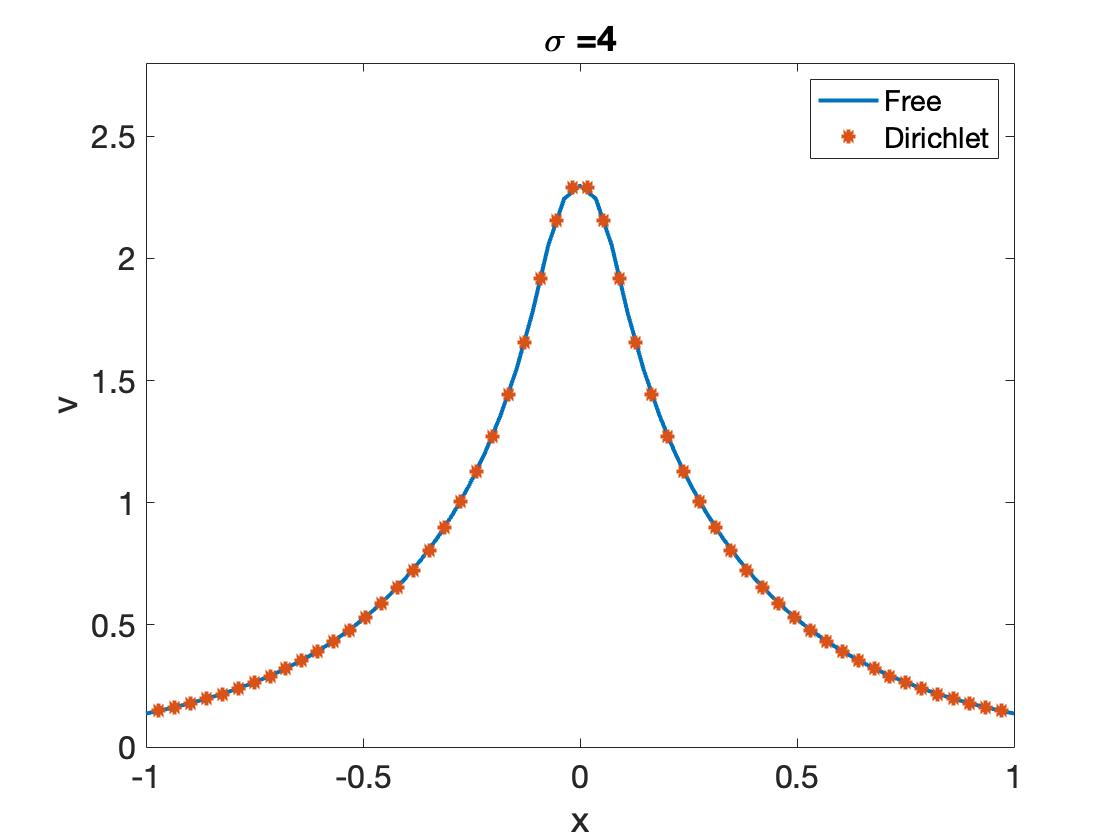}  
 \includegraphics[width=2.75in]{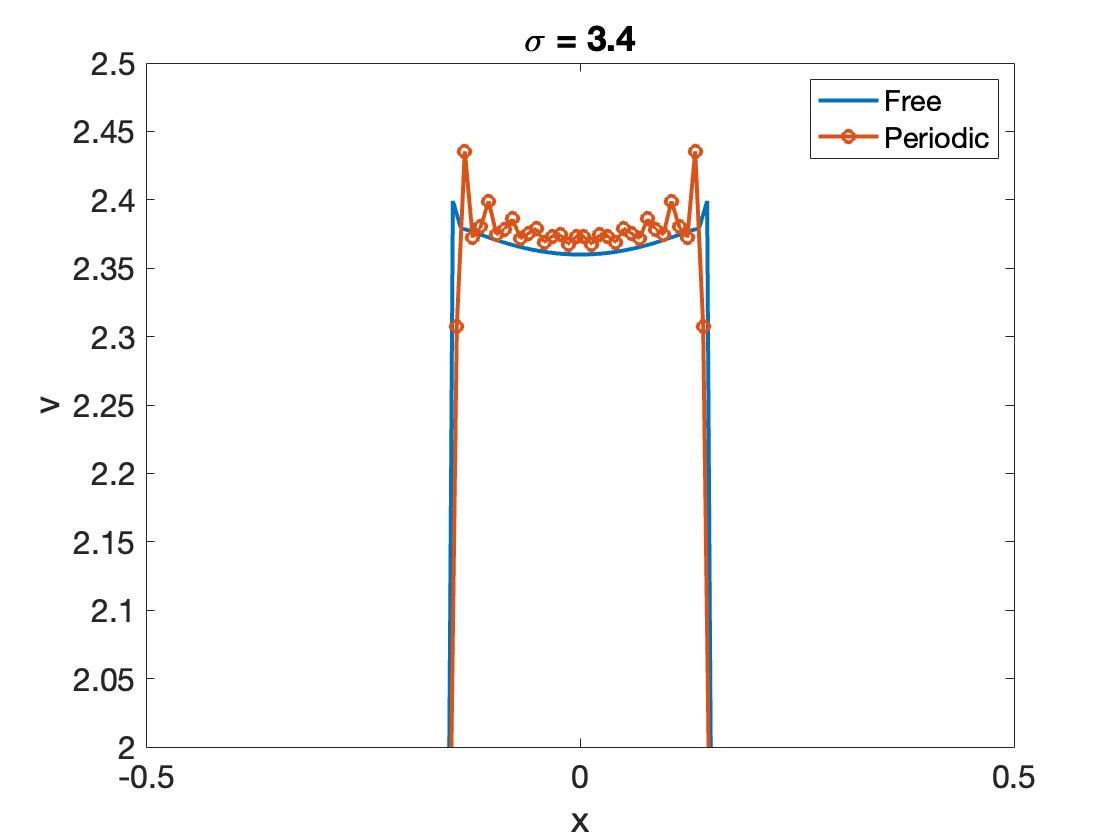}
\caption{Pulse solutions of nonlocal Gray-Scott model \eqref{e:GS} with map $K$ defined by exponential kernel with $\sigma =4$ (left column), and $\sigma =3.4$ (right column). Simulations compare nonlocal 'free' boundary constraints with periodic boundary conditions (FFT) and use $M = 2^{13}$ nodes.}
\label{fig:Exponential_RL_Periodic}
\end{figure}

\begin{figure}[ht] 
   \centering
 \includegraphics[width=2.75in]{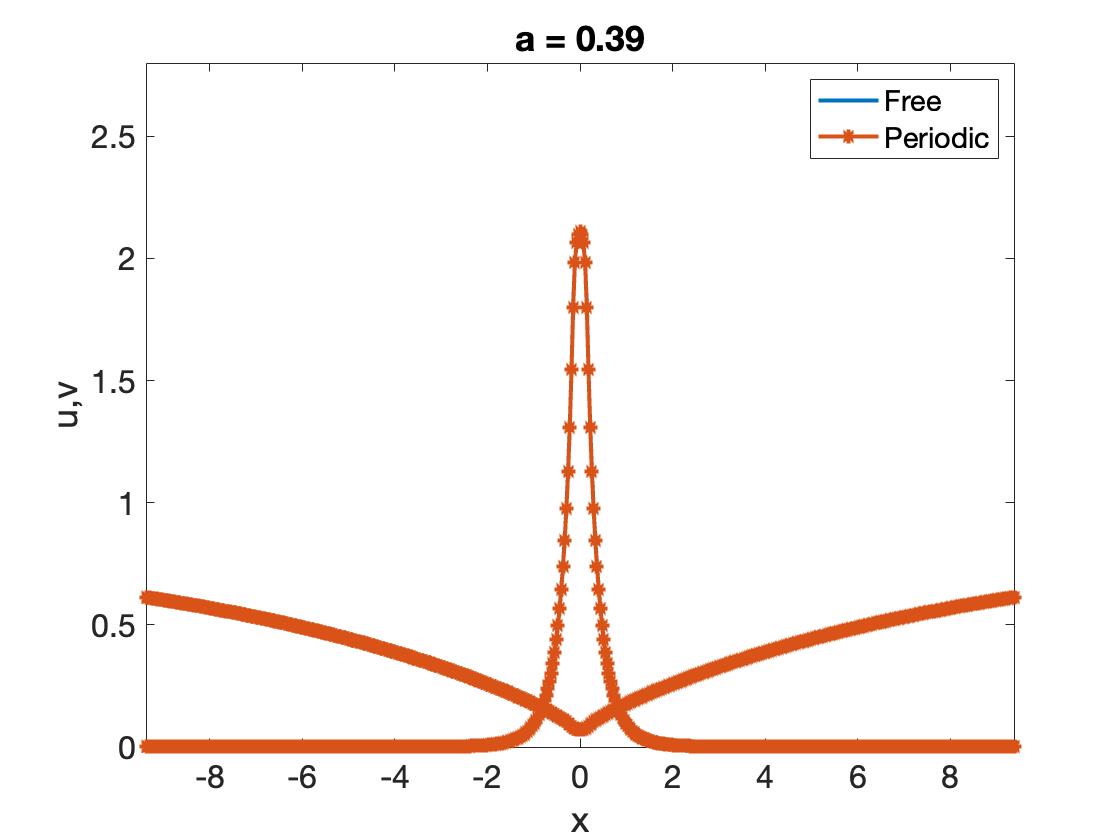}
  \includegraphics[width=2.75in]{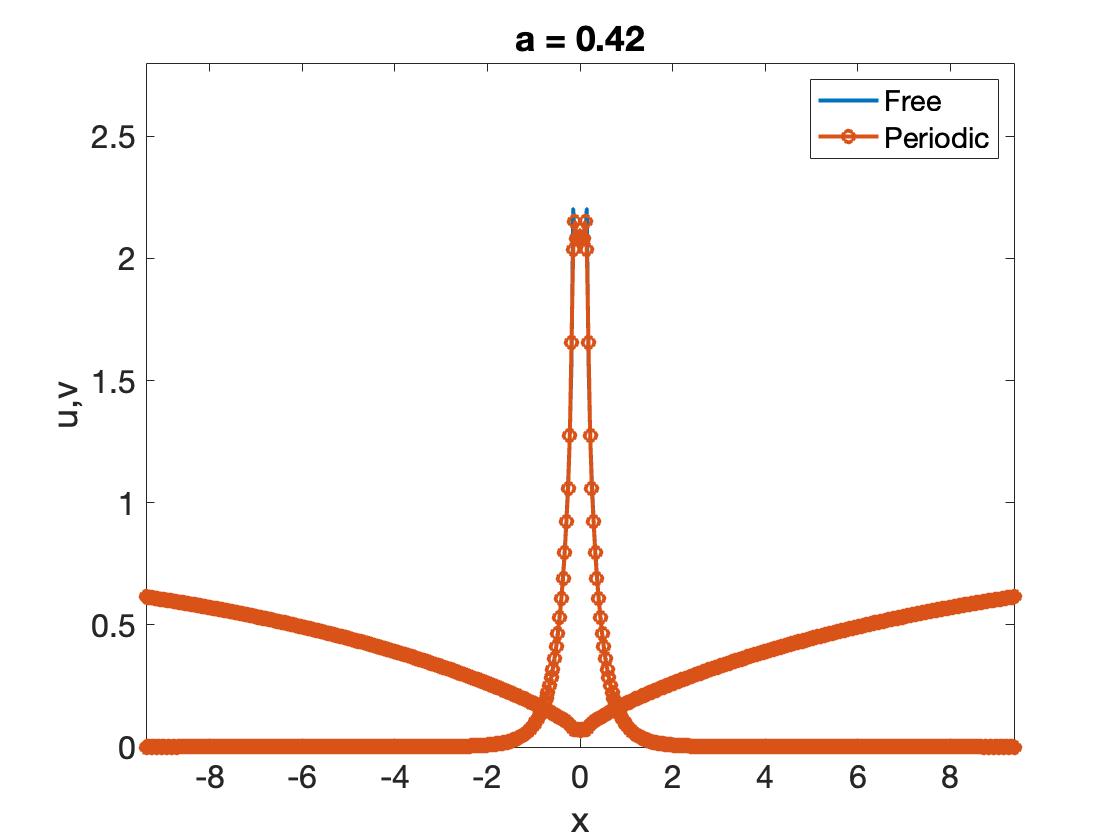}
 
 \includegraphics[width=2.75in]{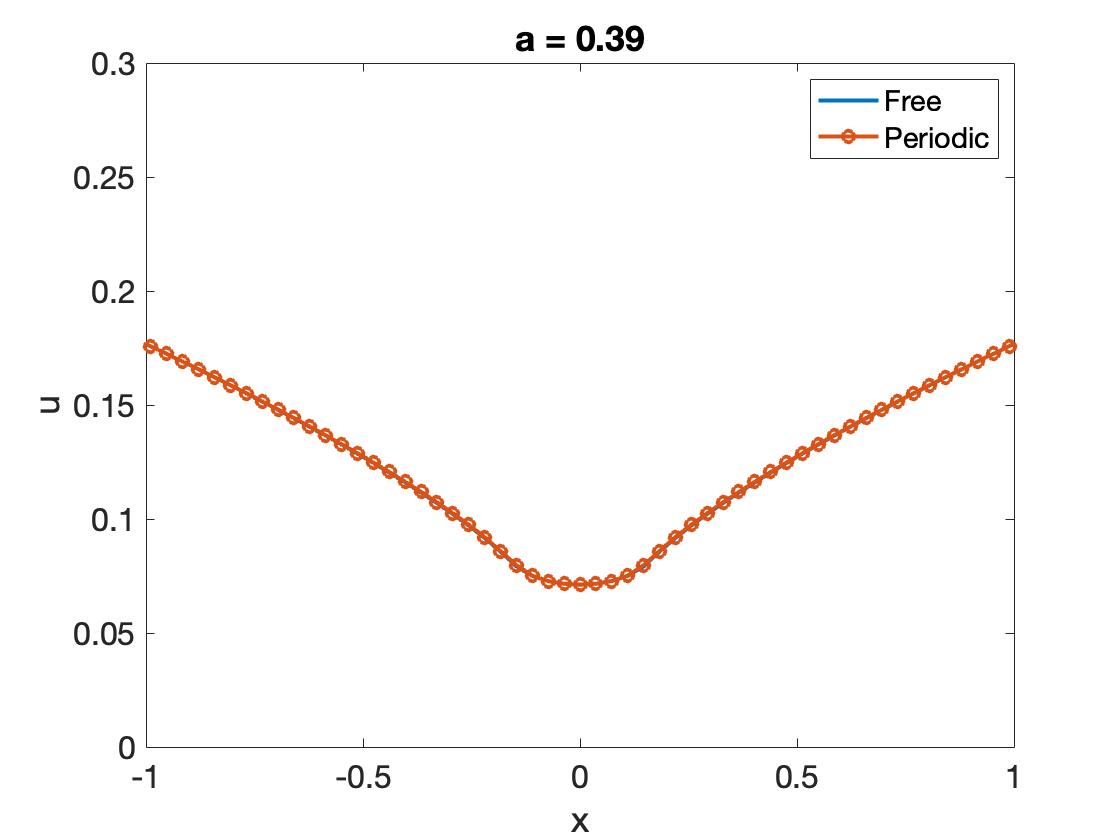}
 \includegraphics[width=2.75in]{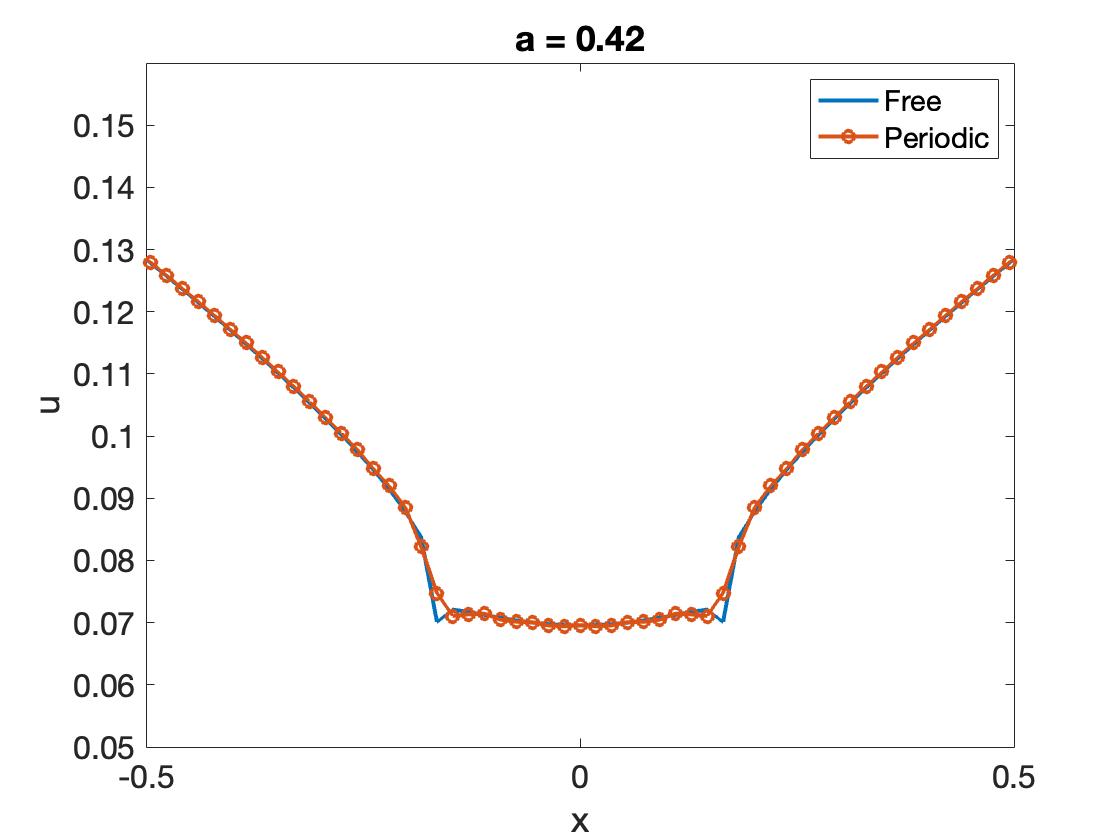}
 
 \includegraphics[width=2.75in]{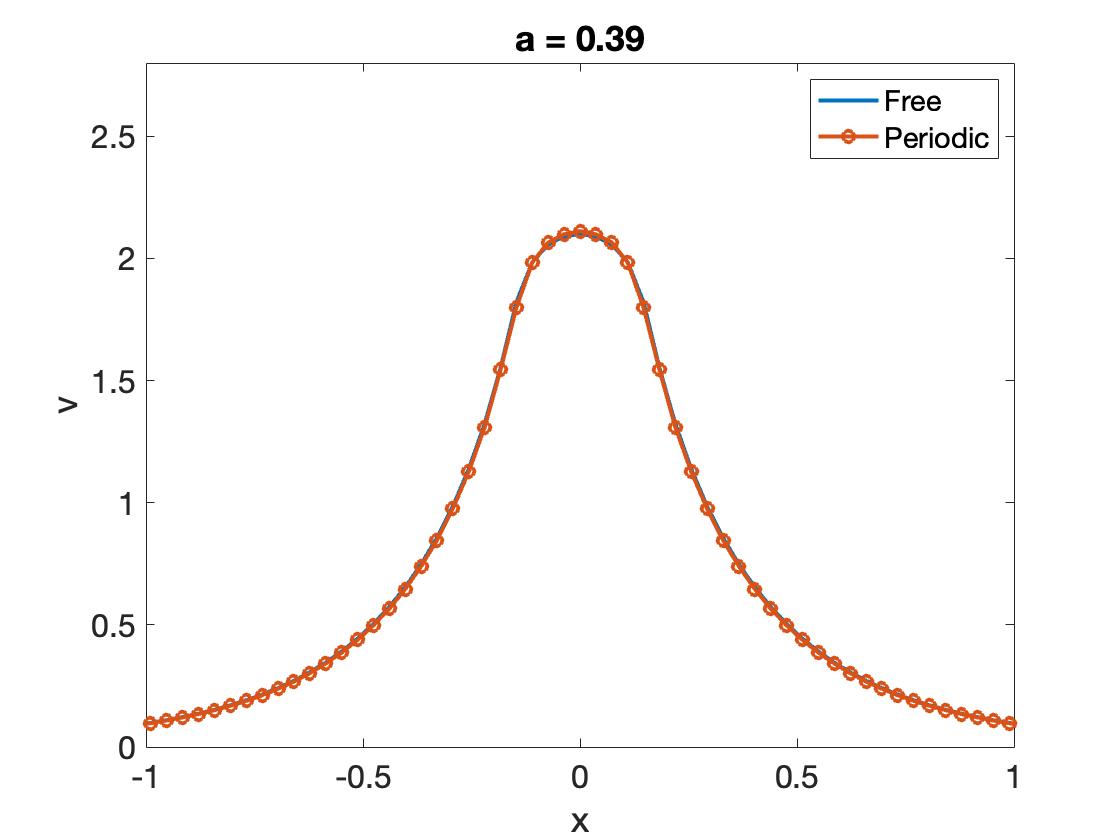}
 \includegraphics[width=2.75in]{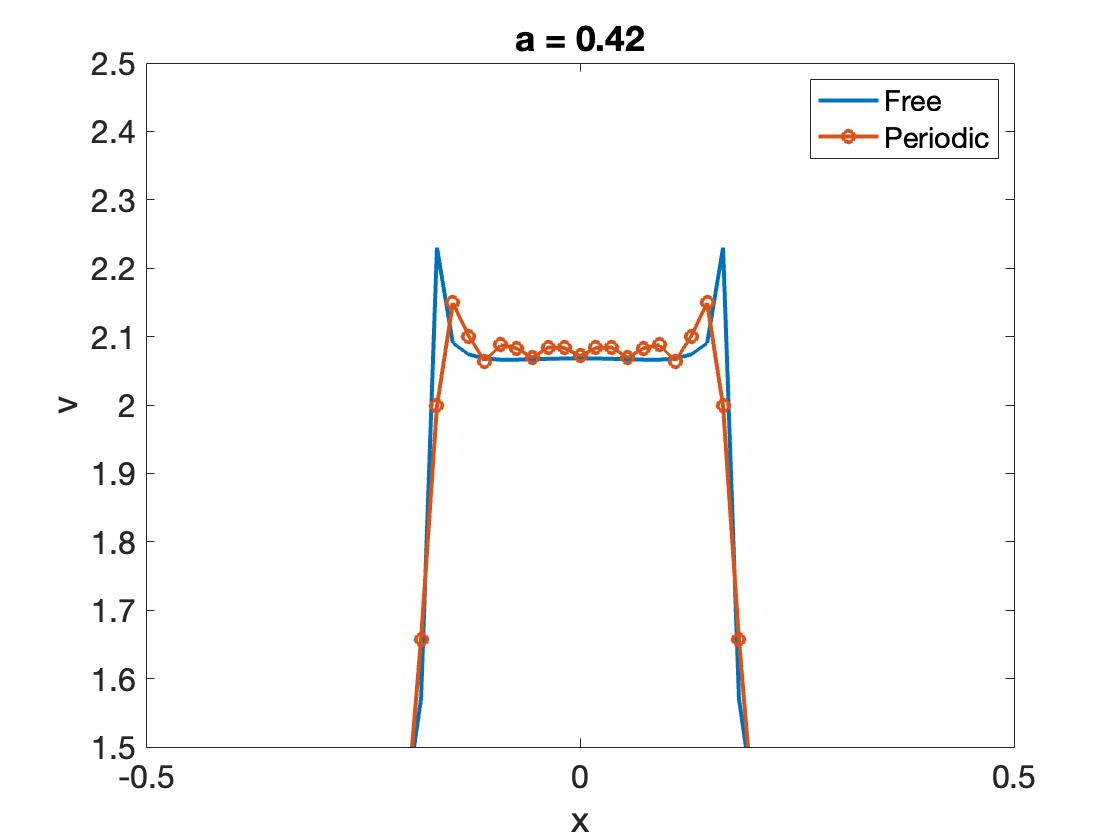}
\caption{Pulse solutions of nonlocal Gray-Scott model \eqref{e:GS} with map $K$ defined by algebraic kernel with $a =0.39$ (left column), and $a =0.42$ (right column). Simulations compare nonlocal 'free'  boundary constraints with periodic boundary conditions (FFT) and use $M = 2^{13}$ nodes.}
\label{fig:Algebraic_RL_Periodic}
\end{figure}

In our simulations we also compared the effects of changing the length of the computational domain. In Figures \ref{fig:Exponential_CompareLength} we plot pulse solutions obtained using  the exponential kernel with domains $\Omega_1 = [-75/4, -75/4]$ and $\Omega_2=[ -25,25]$ in the case of Dirichlet and 'free' boundary constraints, and with $\tilde{\Omega}_1 = [-75/4, -75/4]$ and $\tilde{\Omega}_2=[ -25,25]$ in the case of Neumann boundary constraints. In all cases, we assume
$\sigma =3.4$ and the same parameters and initial conditions as in \eqref{eq:continuation_IC}.
In Figure \ref{fig:Exponential_CompareLength} we have a similar set up, but now use the algebraic kernel with $a= 0.42$. 
The plots show that bigger domains generate shorter pulses in all cases. Not surprisingly, larger domains also allow for the solution to  approach more closely the limit $(u,v) = (1,0)$, which is expected for  pulses in the local Gray-Scott model.

\begin{figure}[ht] 
   \centering
 \includegraphics[width=2.75in]{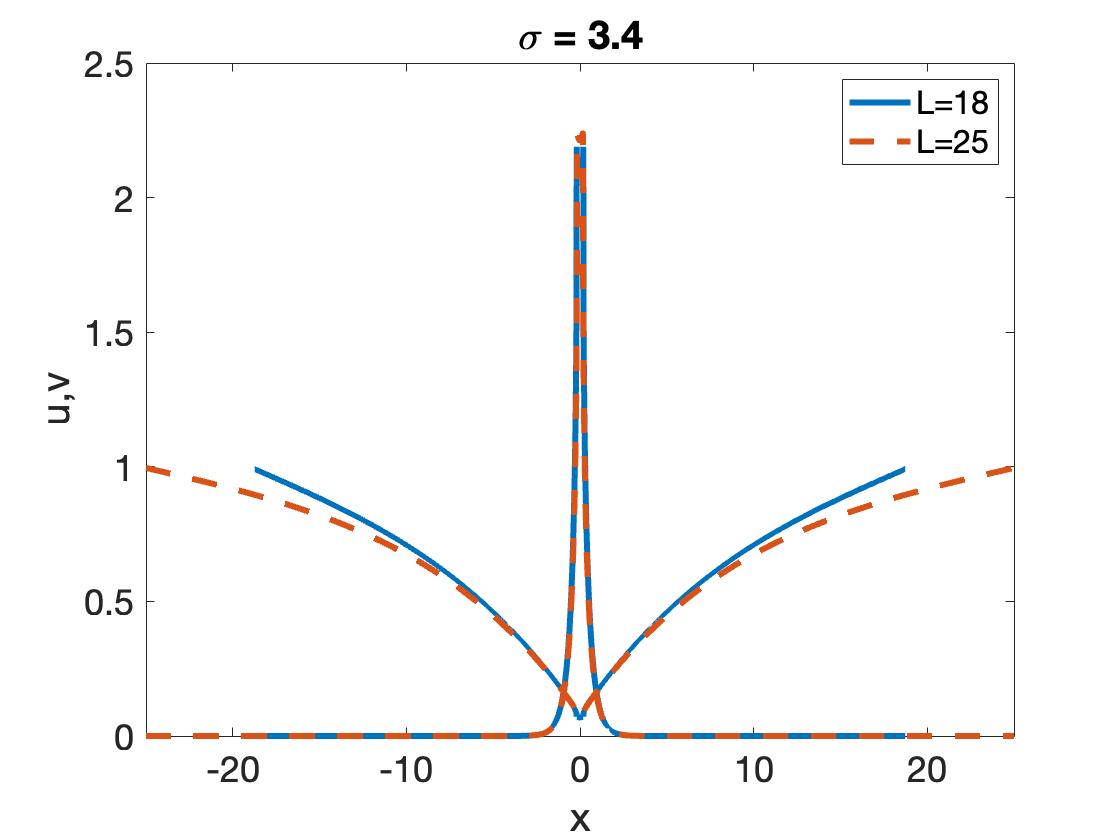}
 \includegraphics[width=2.75in]{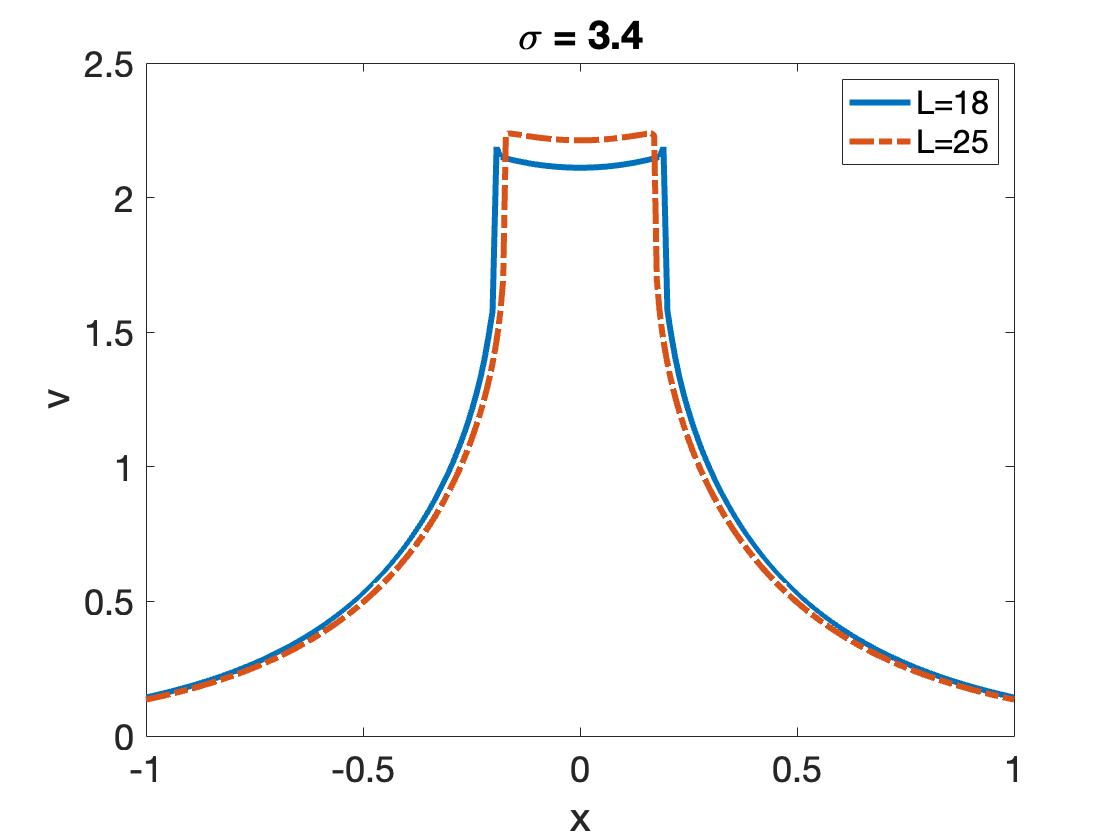}
 \includegraphics[width=2.75in]{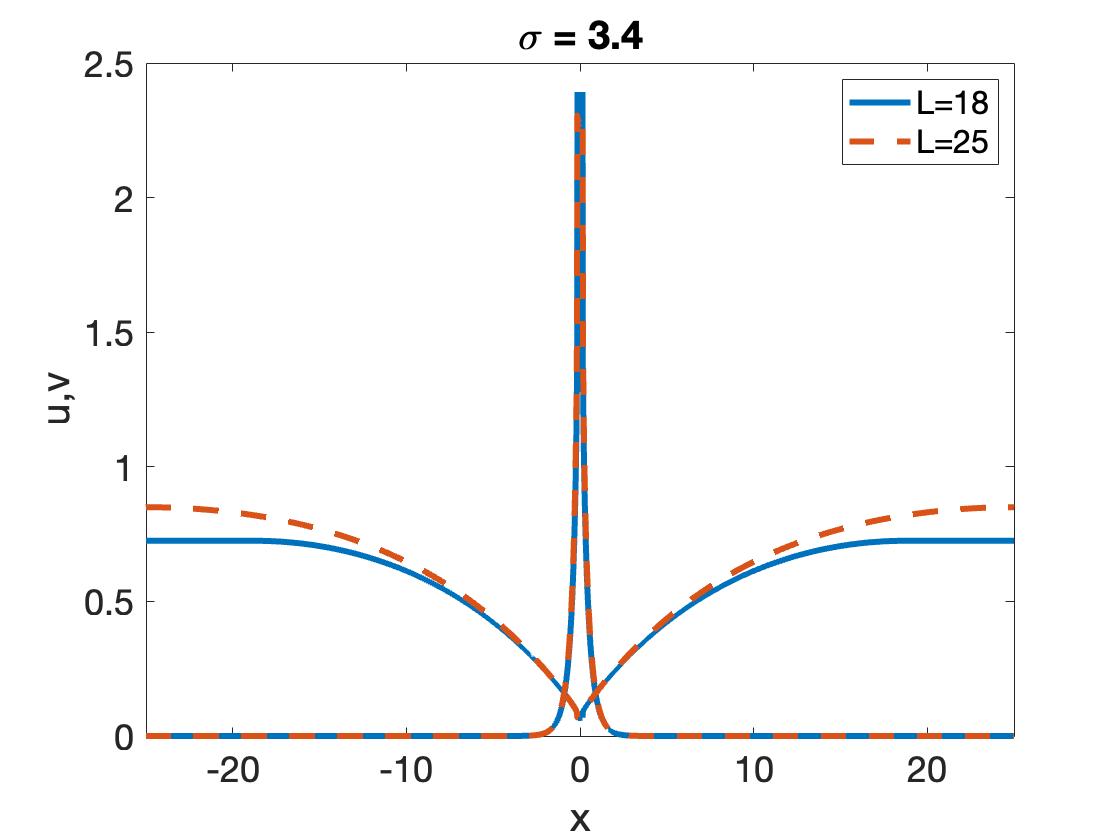}
 \includegraphics[width=2.75in]{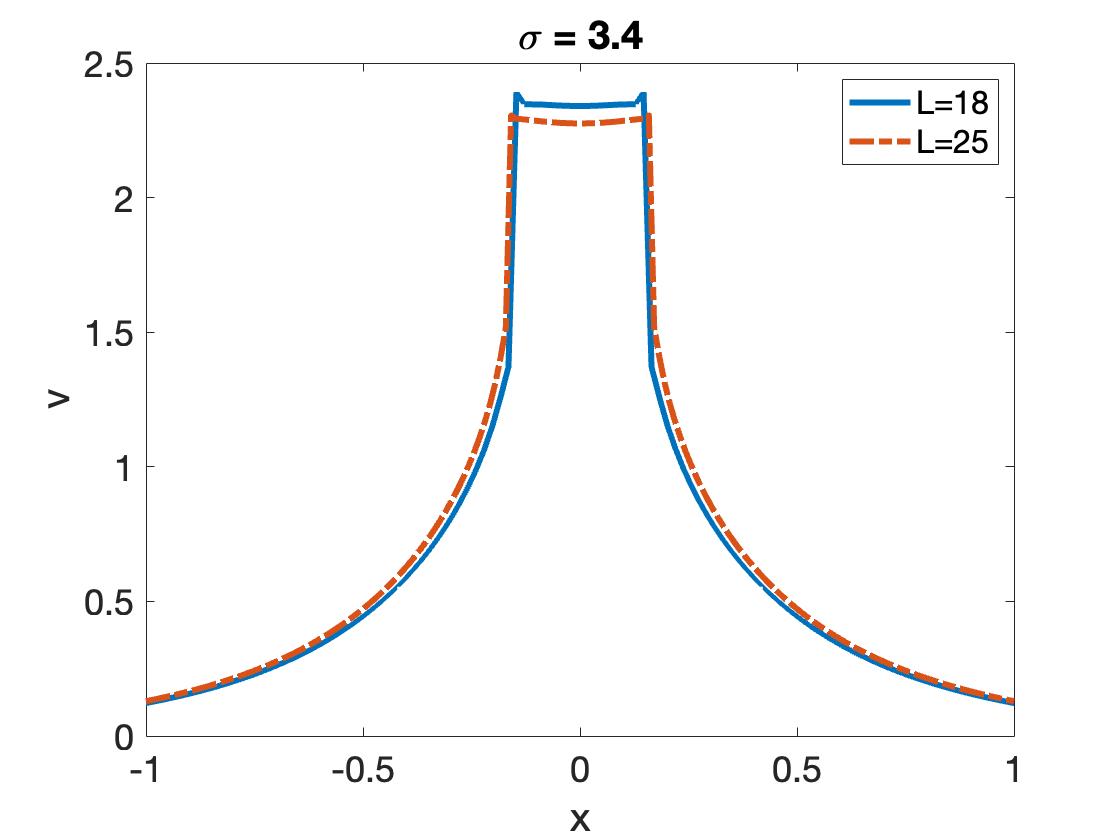}
 \includegraphics[width=2.75in]{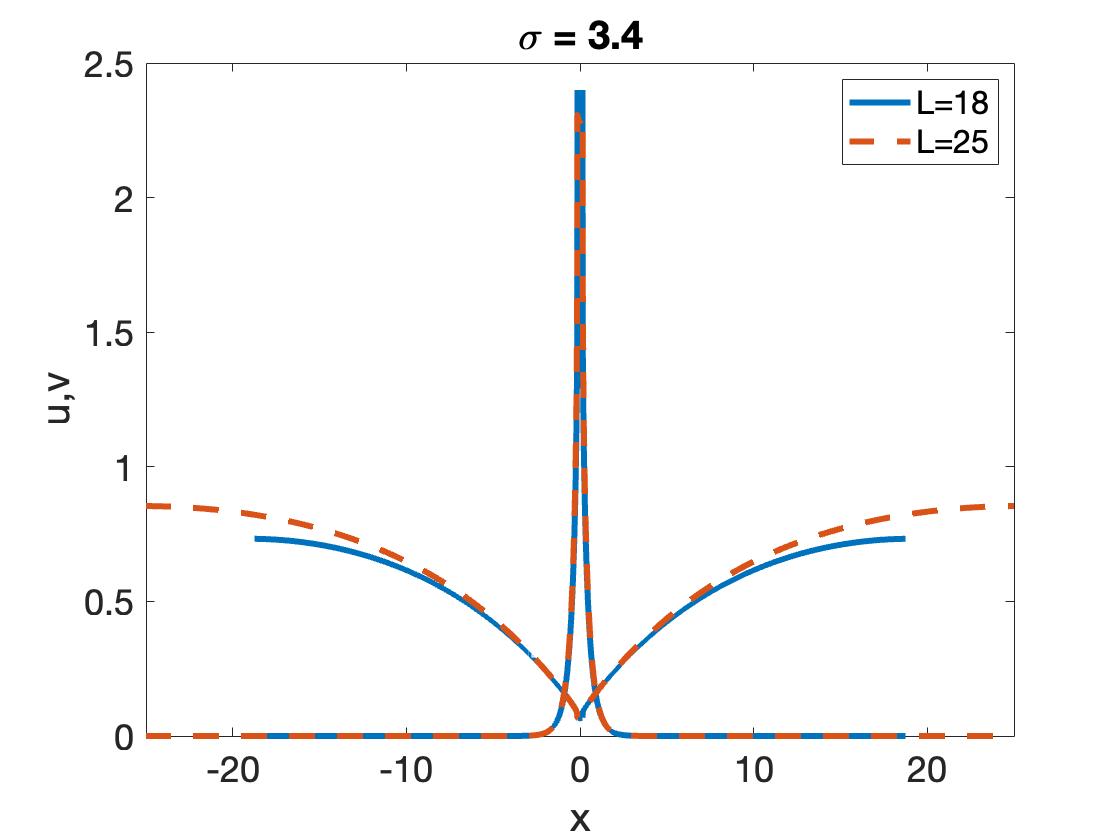}
 \includegraphics[width=2.75in]{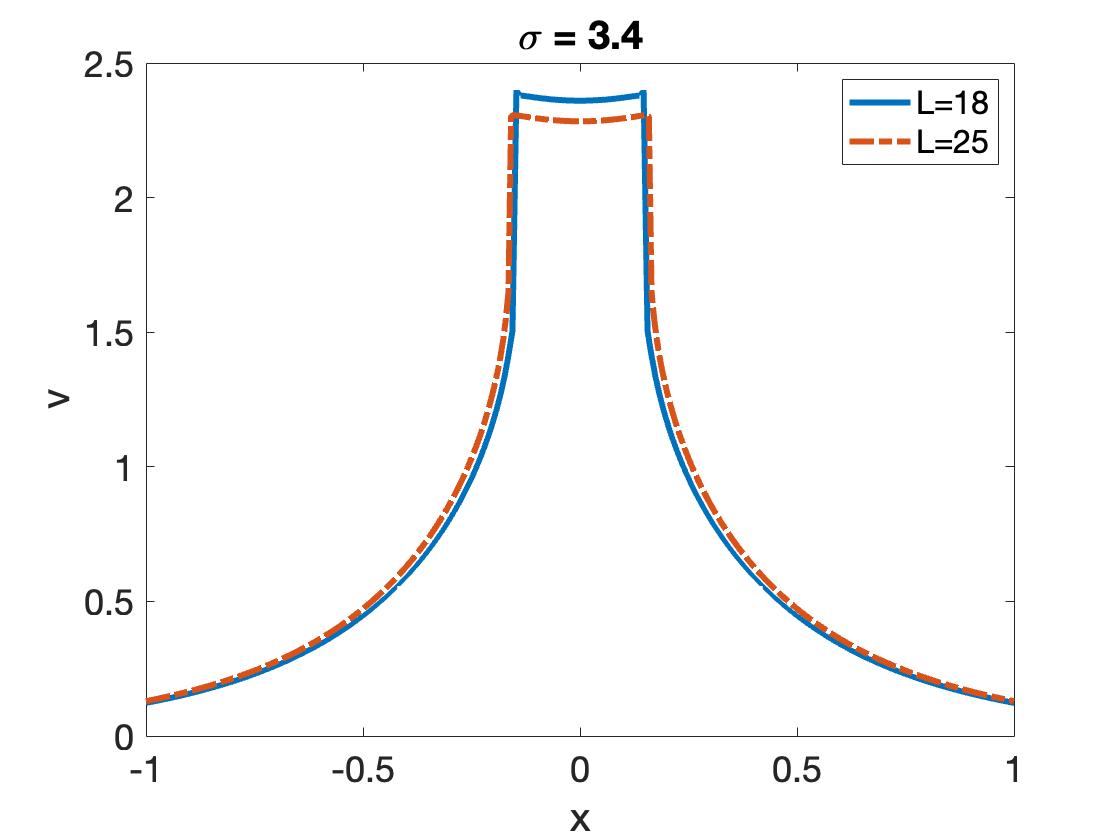}
\caption{Pulse solutions of nonlocal Gray-Scott model \eqref{e:GS} with map $K$ defined by exponential kernel with $\sigma =3.4$ and nonlocal Dirichlet (top row), Neumann (middle row), and `free' boundary constraints (bottom row).  Simulations compare effect of domain size with $\Omega =[-75/4,75/4]$ and $M =2^{12}$ (solid lines), and $\Omega =[-25,25]$ $M=2^{13}$ (dashed line). }
\label{fig:Exponential_CompareLength}
\end{figure}

\begin{figure}[ht] 
   \centering
 \includegraphics[width=2.75in]{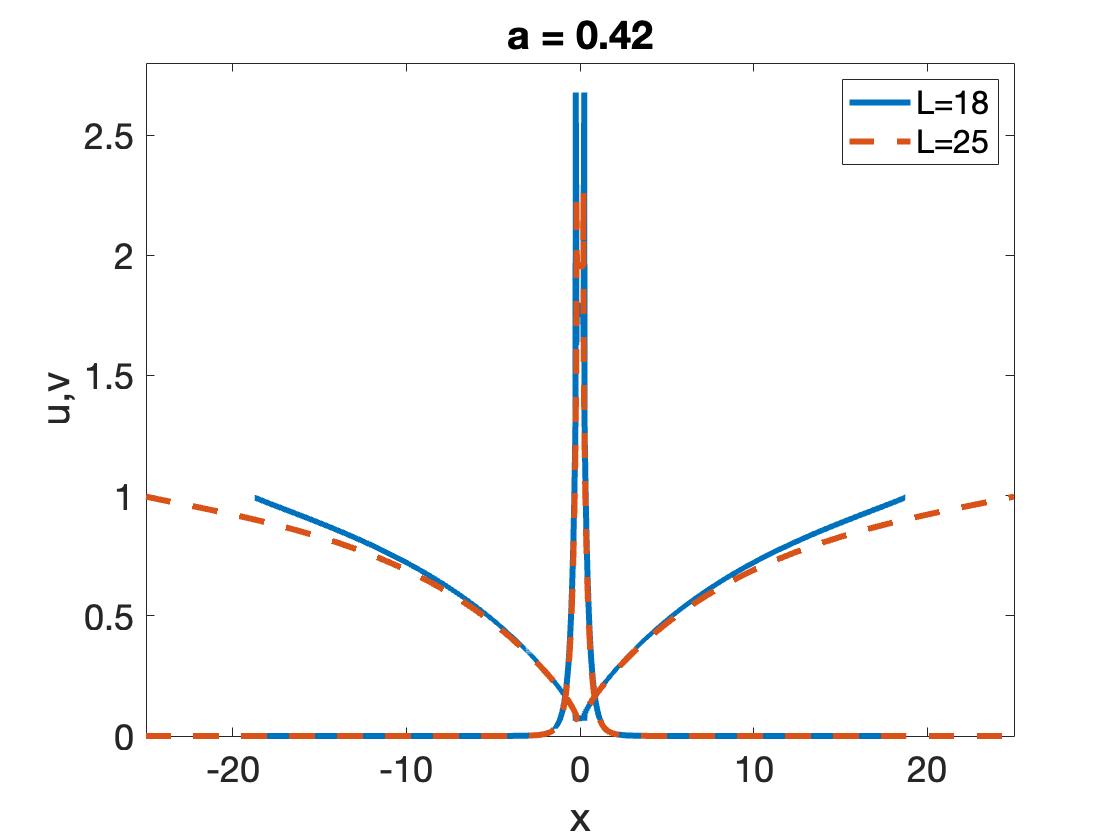}
 \includegraphics[width=2.75in]{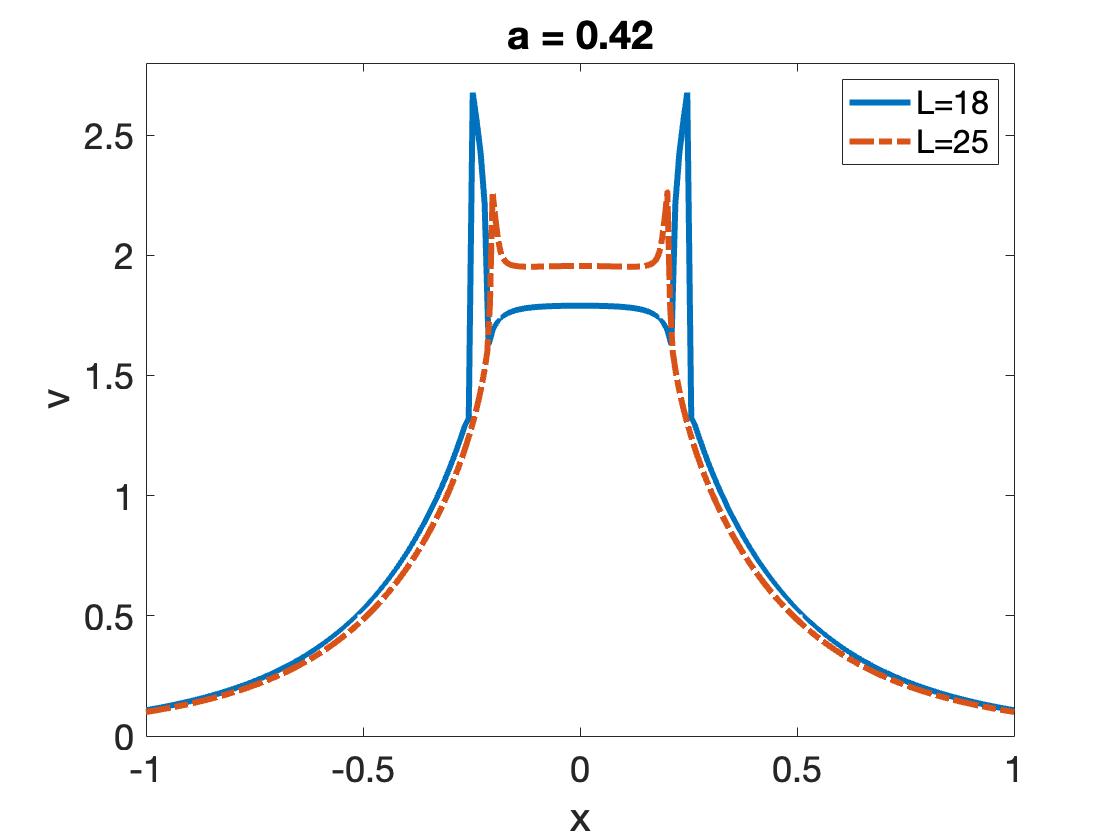}
 \includegraphics[width=2.75in]{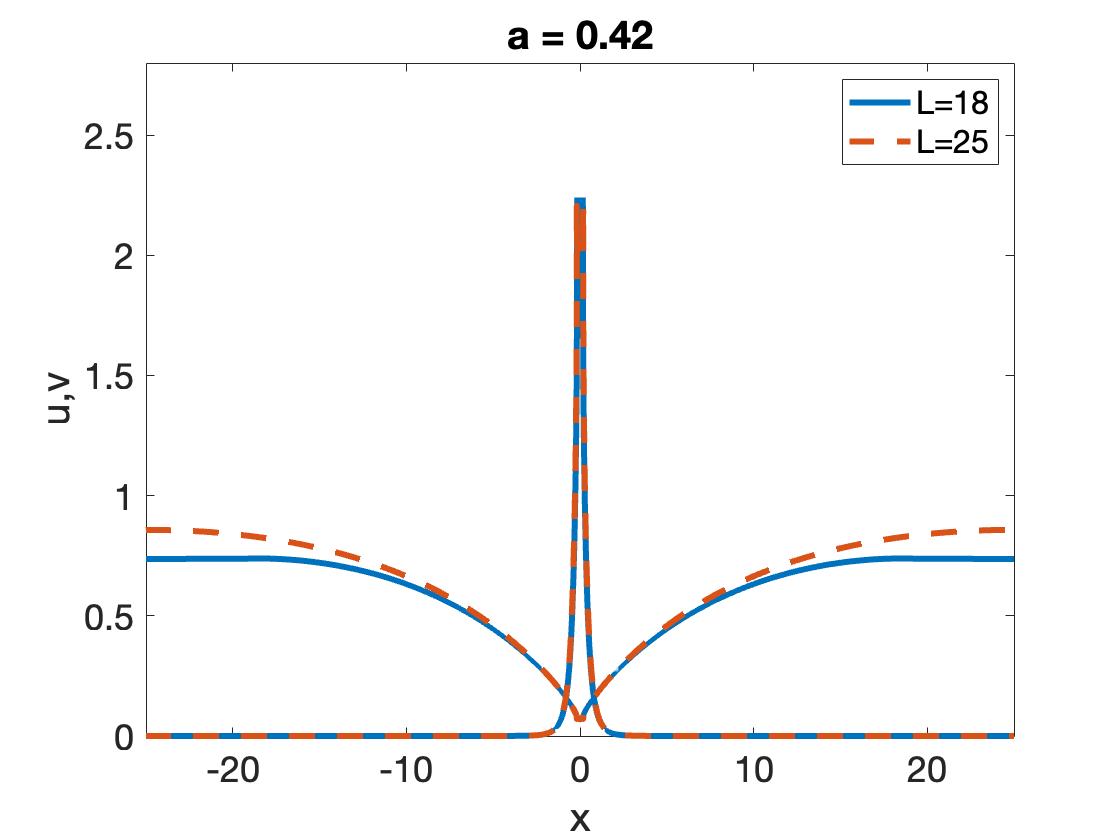}
 \includegraphics[width=2.75in]{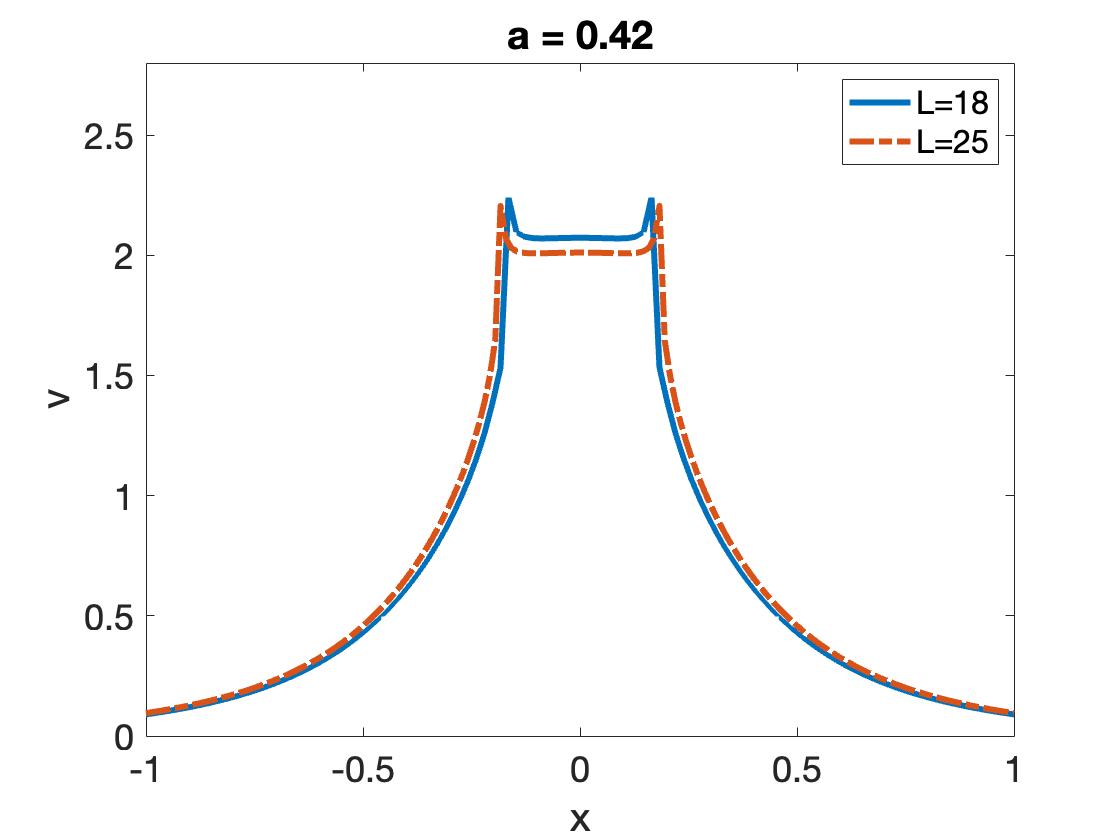}
 \includegraphics[width=2.75in]{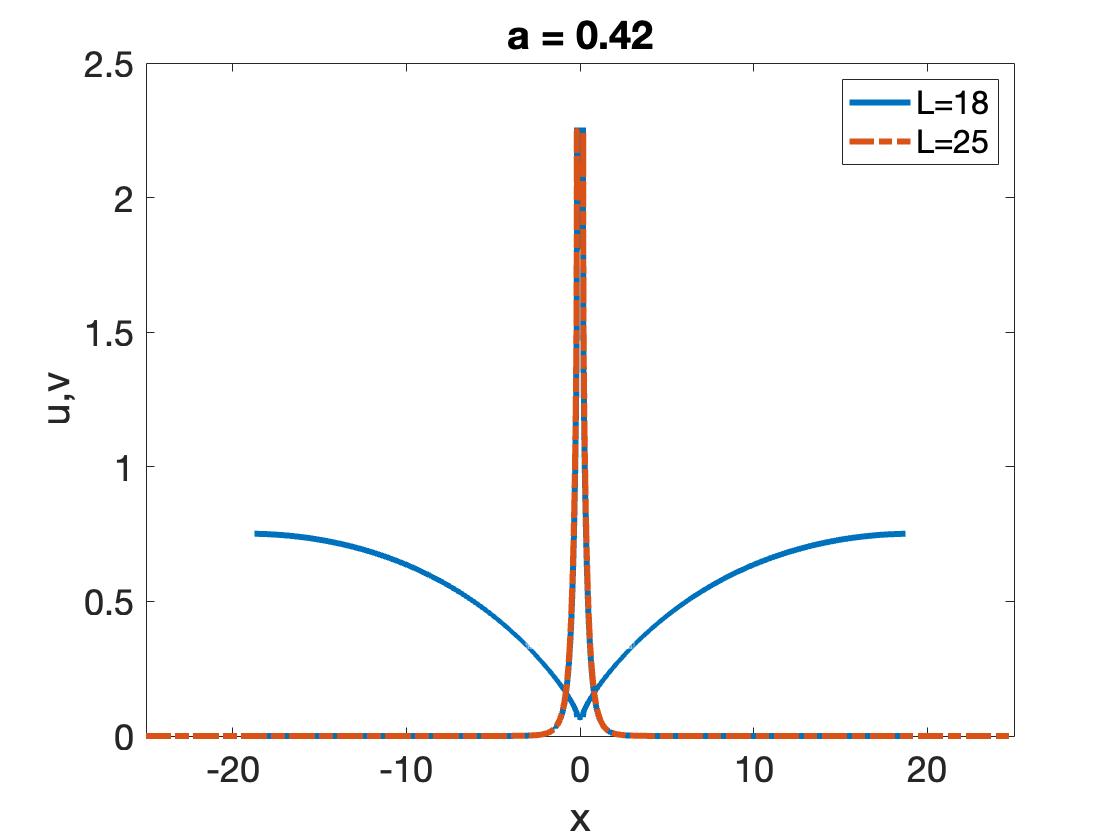}
 \includegraphics[width=2.75in]{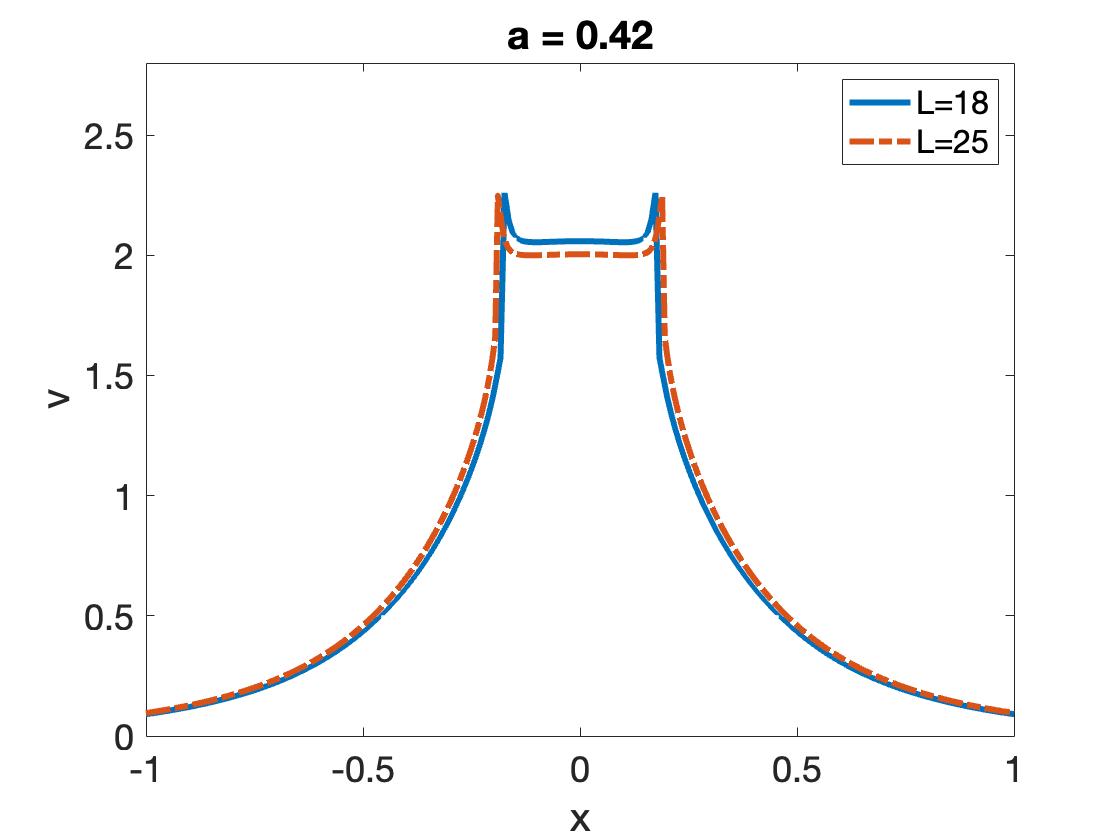}
\caption{Pulse solutions of nonlocal Gray-Scott model \eqref{e:GS} with map $K$ defined by algebraic kernel with $a=0.42$ and nonlocal Dirichlet (top row), Neumann (middle row), and `free' boundary constraints (bottom row).  Simulations compare effect of domain size with $\Omega =[-75/4,75/4]$ and $M =2^{12}$ (solid lines), and $\Omega =[-25,25]$ $M=2^{13}$ (dashed line). }
\label{fig:Algebraic_CompareLength}
\end{figure}

Finally, unless otherwise stated, all
simulations are performed with $M=2^{13}=8192$ quadrature nodes and a time step $dt\approx 1.5 E -4$. The tolerance $tol$ in the algorithms \ref{al:dirichlet}-\ref{al:neumann} is set to $10^{-8}$. The following section shows that the above spatial and temporal discretization guarantee that our solutions converged, thus the resulting pulse solutions and the different profiles obtained between nonlocal boundary conditions are not due to numerical artifacts.

\section{Convergence}\label{s:continuation_cvg}
In this section, we numerically check that the profiles obtained in the previous section with $M=2^{13} (=8192)$ quadrature nodes and a time step $dt\approx 1.5E-4$  converge, so that they are not affected by numerical noises when the tolerance of Algorithms \ref{al:dirichlet} and \ref{al:neumann} is set to $10^{-8}$.

To do this, we perform a series of convergence tests using the same setup studied in Section \ref{s:continuation}. Therefore, we set the parameters of nonlocal Gray-Scott model as follows
\[d_u = 1, \quad d_v= 0.01, \quad  A = 0.01, \quad B = (0.01)^{1/3}/2\]
and the initial conditions are defined using \eqref{eq:continuation_IC}.

For all three types of boundary constraints,
we perform a series of six simulations with $M =[2^9, 2^{10}, 2^{11}, 2^{12},2^{13}, 2^{14}]$ nodes with corresponding time steps $dt = [0.0025,$ $0.00125,$ $0.000625,$ $0.0003125,$ $0.00015625,$ $0.000078125]$ for all kernels considered, meaning exponential kernels with $\sigma\in \{3.4,4\}$ and algebraic kernels with $a\in\{0.39,0.42\}$. We refer to the previous section for the kernels' definitions.

For the nonlocal Dirichlet and the `free' boundary constraint case we use a physical domain $\Omega =[-75/2,75/2]$, while in the Neumann case we use $\tilde{\Omega} = [-75/4 -75/4]$ with a computational domain $\Omega =[-75/2, 75/2]$ for the operator $K$.

As no exact solutions are available, to determine the convergence of our algorithms we track the $L^1$ error between the solutions $(u_M, v_M)$, obtained with $M = [2^9, 2^{10}, 2^{11}, 2^{12}, 2^{13}]$ nodes, and the solution $(\hat{u}, \hat{v})$ obtained with $M =2^{14}$ nodes. Then, the order of convergence between two approximation $(u_M,v_M)$ and $(u_{M+1}, v_{M+1})$ can be  computed using
\[ \mbox{order}_u = \log \left[ \frac{ \| u_{M+1}- \hat{u} \|_{L^1} }{ \| u_{M}- \hat{u} \|_{L^1} } \right] /    \log \left[ \frac{h_{M+1}}{h_M} \right], \]
\[ \mbox{order}_v = \log \left[ \frac{ \| v_{M+1}- \hat{v} \|_{L^1} }{ \| v_{M}- \hat{v} \|_{L^1} } \right] /    \log \left[ \frac{h_{M+1}}{h_M} \right]. \]

To keep it concise, we summarize some of these results in two tables that can be found in Appendix. Note that we only report $L^1$ error in both tables, and not $L^2$ as done for the test with manufactured solutions in Section \ref{s:numericalmethods}. This choice of norm is motivated by the profile of the solution for the cases $\sigma=3.4$ and $a=0.42$ that lead to solutions with a strong gradient near the endpoints of the \emph{mesa} and \emph{cat-ear} profiles as shown in Figures~\ref{fig:Exponential_Nonlocal_BC}-\ref{fig:Algebraic_Nonlocal_BC}. As approximating a solution that presents discontinuities limits the order of convergence to half in $L^2$ norm and to one in $L^1$ norm, we decided to present the errors in $L^1$ norm and show that they recover a convergence order equal to one.

Indeed, first we show the convergence errors that we obtain when nonlocal Neumann boundary constraints are enforced in Table~\ref{tab:kernel_exp_alg_neumann_AB2_L1_H_eq_30_DT}. As expected, we recover a convergence order that is at least one on average for the four setups considered ($\sigma= 3.4,4$ and $a=0.39,0.42$). We note that the last order of convergence can sometimes be very large ($\sigma=3.4$ and $a=0.42$). This superconvergence behavior is due to the lack of exact solution which leads us to use our approximation with $M=2^{14}$ quadrature nodes to compute the error obtained with $M=2^{13}$ quadrature nodes.

Second, we show another series of convergence tests in Table~\ref{tab:kernel_algebraic_dirichlet_real_AB2_L1_H_eq_30_DT}
when we impose nonlocal Dirichlet or 'free' boundary constraints. We only display the most "extreme" cases that lead to pulse solutions with \emph{mesa} and \emph{cat-ear} profiles, i.e. the exponential kernel with $\sigma=3.4$ and the algebraic kernel with $a=0.42$. The algorithm has a similar behavior than in the nonlocal Neumann setup, meaning that we observe an overall convergence rate of one or more. We note that for the 'free' boundary constraints, the error in $v$ tends to stagnate between $M=2^{13}$ and $M=2^{14}$ which can be explained by an earlier quadratic convergence when we use $M=2^{12}$ quadrature nodes. 

Overall, our tests show that the Algorithms \ref{al:dirichlet}-\ref{al:neumann} used in Section \ref{s:continuation} to find steady-state pulse solution converge when refining the mesh size and the time step. Moreover, the final errors in $L^1$ norm between the approximation obtained with $M=2^{13}$ and $2^{14}$ quadrature points are of order $10^{-6}$ or smaller for all setups considered. It confirms that our simulations are accurate and do not require a finer mesh size or time step.

\section{Existence of Pulse Solutions}\label{s:existence}

The numerical simulations in the previous section confirm that nonlocal diffusive effects may be able to generate novel patterns in nonlocal reaction-diffusion systems. In the particular case of the Gray-Scott equations we find that increasing the width of the kernel $\nu$, as modeled by the parameters $\sigma$ or $a$, can create pulse solutions with \emph{mesa} profiles (exponential kernel) or \emph{cat-ear} profiles  (algebraically decaying kernel). To partially justify the formation of these patterns we look at the case when the operator $K$ is defined by the exponential kernel since, as shown below, one can obtain an equivalent PDE.

When $\sigma$ is large the operator behaves like the Laplacian and, as seen in the simulations, the   nonlocal equations  support smooth pulse solutions. One can justify the existence of these stationary patterns  using a perturbation argument. 
Indeed, letting $\eps = 1/\sigma$, one can formally write $Ku = (1-\eps^2 \Delta)^{-1} \Delta$. Pre-conditioning the steady state equations by the invertible operator $(1-\eps^2 \Delta)$ leads to the quasilinear  system,
\begin{equation}\label{e:quasilinear}
\begin{split}
0 & = (d_u+ \eps^2 A) \Delta u + A(1-u) - uv^2 + \eps^2 \Delta(uv^2)\\
0 & = (d_v + \eps^2 B) \Delta v -Bv  + uv^2 - \eps^2 \Delta(uv^2).
\end{split}
\end{equation}

If $\eps^2$ is small  we can ignore the last term, $\eps^2\Delta(uv^2)$, in both equations, and recover the local Gray-Scott model, albeit with slightly different diffusive coefficients. It is well known that this local system supports pulse solutions in the regime considered here, i.e. $\delta^2 = (d_v + \eps^2 B), $ $A = \delta^2 a, $ $ B = \delta^\beta b$, with $a,b \sim \rmO(1)$ and $0< \beta <1$, $0<\delta<<1,$ see \cite{arjen1997}. We can therefore consider solutions to \eqref{e:quasilinear}, of the form 
\[ u = u_* + \tilde{u}, \quad v= v_* + \tilde{v},\]
where $(u_*,v_*)$ are the pulse solution of the local system (ignoring the terms $\eps^2 \Delta(uv^2)$ in \eqref{e:quasilinear}), and the perturbation $U = (\tilde{u},\tilde{v})$ is the new unknown.  In this section we refer to $(u_*,v_*)$ as the 'local approximation'.

The resulting system can be written abstractly as
\[ U =LU + \tilde{F}(U,\eps)\]
with $L: H^2(\R) \times H^2(\R) \longrightarrow L^2(\R) \times L^2(\R)$ defined by
\[ LU = \begin{bmatrix}
(d_u+ \eps^2 A) \Delta  - A & 0 \\
0 & (d_v + \eps^2 B) \Delta  -B 
\end{bmatrix}
\begin{bmatrix}
\tilde{u} \\ \tilde{v}
\end{bmatrix} +
 \begin{bmatrix}
-v_*^2 & -2u_*v_*^2\\
v_*^2 &  2u_*v_*^2
\end{bmatrix}
\begin{bmatrix}
\tilde{u} \\ \tilde{v}
\end{bmatrix}
\]
Because $v^*$ decays exponentially to zero, we find that $L$ is a compact perturbation of an invertible operator, and it is therefore Fredholm with index zero. One can therefore use Lyapunov-Schmidt reduction to justify the existence of solutions $U$ to equation \eqref{e:quasilinear} for small values of $\eps$.

On the other hand, when $\eps =1/\sigma$ is large the nonlinear terms $\eps^2 \Delta(uv^2)$ can no longer be ignored. One must therefore treat the system \eqref{e:quasilinear} as a quasilinear system, i.e.
\begin{align*}
\begin{bmatrix} 0 \\0 
\end{bmatrix} = & \underbrace{
\begin{bmatrix}
(d_u+ \eps^2 A + \eps^2 v^2)   & 2 \eps^2 uv \\
- \eps^2 v^2  & (d_v + \eps^2 B- 2 \eps^2 uv) 
\end{bmatrix}}_{a(u,v;\eps)}
\begin{bmatrix} \Delta u \\ \Delta v \end{bmatrix} +
\begin{bmatrix}
-A & 0 \\
0 & -B 
\end{bmatrix} 
\begin{bmatrix} u \\ v \end{bmatrix}\\
& +
\begin{bmatrix}
A - uv^2 \eps^2 ( 4 v u_x v_x + 2u v_x^2) \\
 uv^2 - \eps^2(4v u_x v_x + 2u v_x^2)
\end{bmatrix}.
\end{align*}
If we assume for the moment that the matrix $a(u,v;\eps)$ is non-singular, then we can pre-condition the equations with $a(u,v;\eps)^{-1}$, leading to a system of now elliptic semilinear equations. This suggests it is possible to use barrier functions and comparison principles to prove existence of solutions, even for large values of $\eps$. To check this assumption,
Figure \ref{fig:compare_determinant} computes the determinant of $a(u,v;\eps)$ using the local approximation $(u_*,v_*)$ and the solution to the corresponding nonlocal problem obtained using our numerical algorithms.
The plots show that when using the local approximation, the determinant of $a(u_*,v_*,\eps)$ crosses the horizontal axis when $x \sim 0$, where as the numerical approximation to the nonlocal equation leads to a non-negative determinant. The plots therefore suggest that in order to support pulse solutions when $\eps$ is large (i.e. when width of the kernel $\nu$ is wide), the system must flatten the spike, partially justifying the \emph{mesa} profile seen in the numerical solutions.  We plan to rigorously justify the above arguments in a companion paper. 

\begin{figure}[ht] 
   \centering
 \includegraphics[width=2.75in]{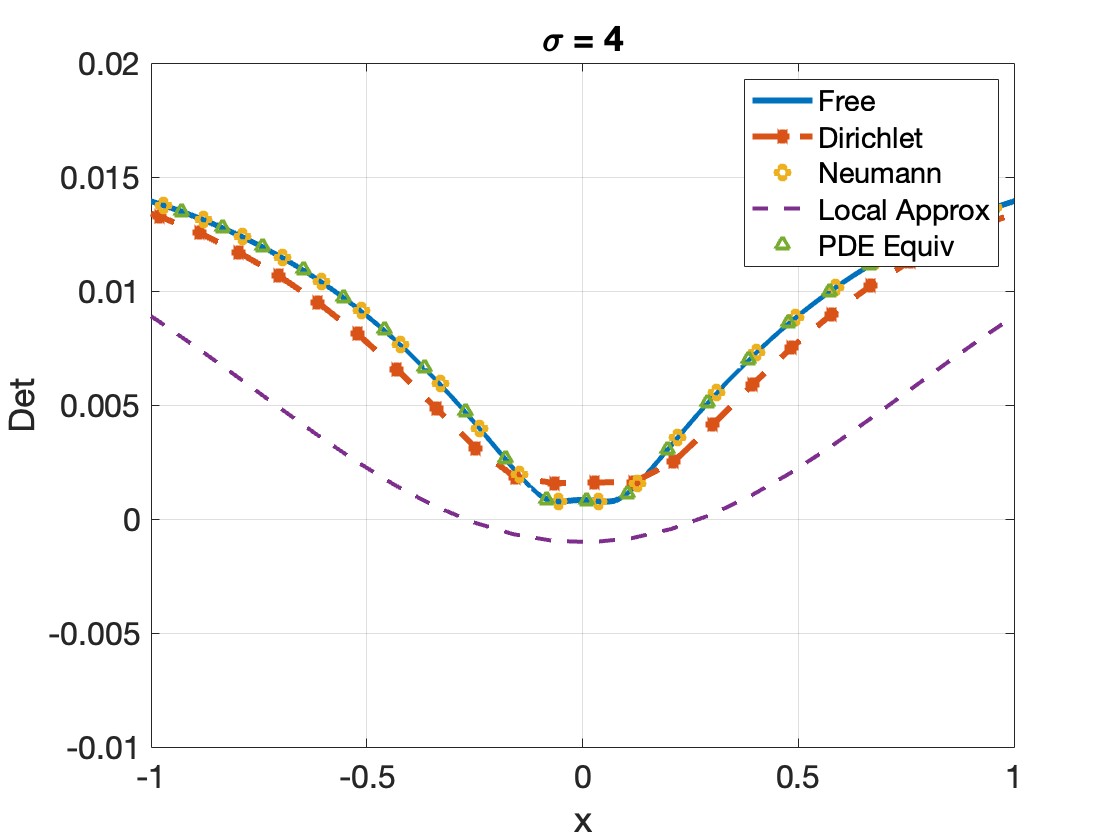}
 \includegraphics[width=2.75in]{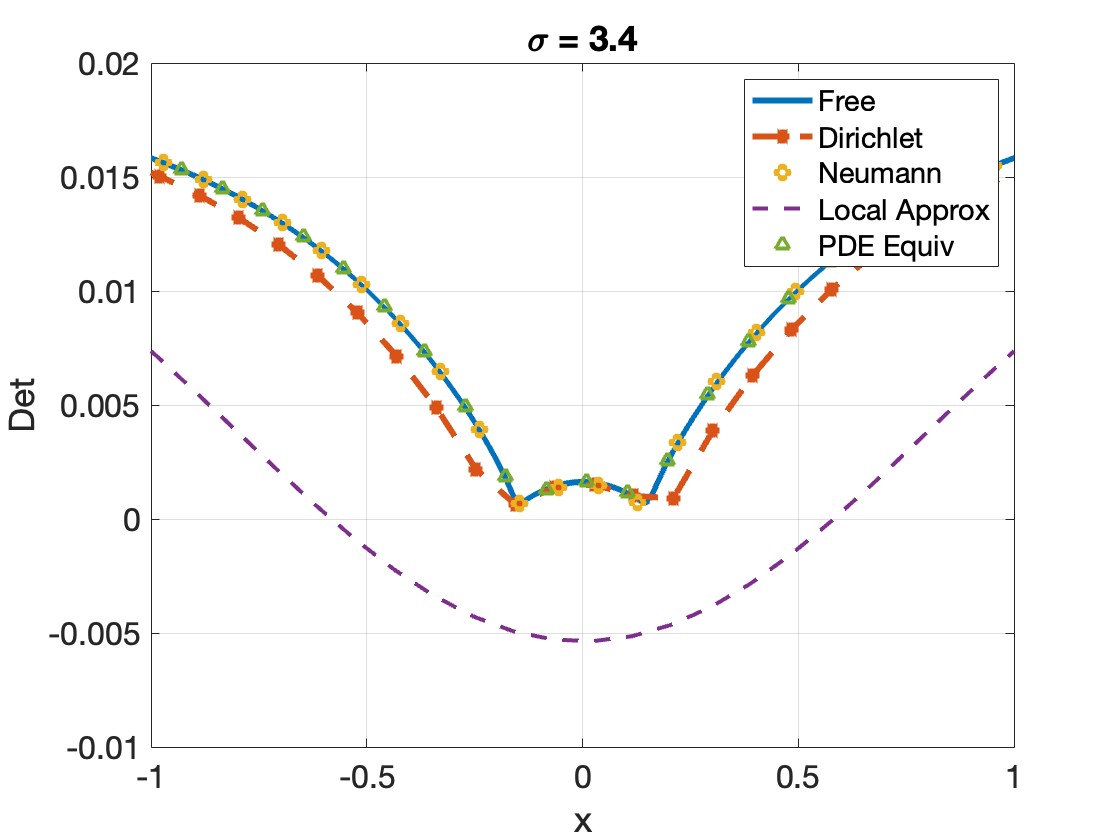}
 
 \caption{Plot of the det$[ a(u,v;\eps)]$ evaluated using the 'local approximation', $(u_*,v_*)$, the numerical solutions of \eqref{e:GS} using our quadrature methods, and the solution of PDE \eqref{e:quasilinear} obtained using Neumann boundary conditions on a domain $\Omega=[-75/4, 75/4]$
 Left plot shows results using $\sigma = 4$, right plot using $\sigma =3.4$. Other parameters are: $d_u = 1, \quad d_v= 0.01, \quad  A = 0.01, \quad B = (0.01)^{1/3}/2$.  }
\label{fig:compare_determinant}
\end{figure}

Finally, we mention that numerical simulations of the PDE \eqref{e:quasilinear} using Neumann boundary conditions do lead to the same patterns seen in simulations of the corresponding nonlocal equations, \eqref{e:GS}, using our quadrature method, see Figure \ref{fig:compare_determinant}.

\section{Conclusion}
\label{s:conclusion}

In this paper we considered the nonlocal system \eqref{e:GS}, previously introduced by the authors in \cite{cappanera2024analysis}.
This model is a generalization of the well-known local Gray-Scott model, a reaction diffusion system known for producing a wealth of interesting solutions, including pulses, periodic solutions, self-replicating patterns, and spatio-temporal chaos \cite{pearson1993}. Equations \eqref{e:GS} use the same nonlinearities as the local system, but replace the Laplacian term with a convolution operator modeling nonlocal diffusion. 

The motivation for studying the integro-differential equations \eqref{e:GS} comes from problems in ecology, where similar nonlocal operators are used to describe the spread of individuals, plant seeds, or infections. In these examples, these integral operators are defined using symmetric and spatially extended convolution kernels, which creates some difficulties when posing the equations on bounded domains. In particular, one needs to impose nonlocal boundary constraints in order to define the operator in a way that is consistent with applications. To properly enforce these constraints when running numerical simulations one can use finite elements or quadrature methods to discretize the equations in space. However, in these formulations the nonlocal operator is then approximated by a dense matrix, which can slow down computations. As a result, most numerical simulations of these, and similar nonlocal equations, are done using Fourier spectral methods. Although this approach generates efficient numerical schemes it also enforces periodic boundary conditions, which may not be appropriate for some problems. The goal of this paper is to analyze if and when different types of boundary constraints, i.e. local vs. nonlocal, lead to different solutions. For simplicity, here we considered only pulse solutions of the 1-d nonlocal Gray-Scott model \eqref{e:GS}.

To run simulations of the nonlocal Gray-Scott model, we used the quadrature method described in \cite{jaramillo2021} for the spatial discretization of the nonlocal operator, combined with a second-order Adams-Bashforth method for the time marching.
In addition to considering nonlocal Dirichlet, nonlocal Neumann, 'free' and periodic boundary constraints, we also study the effect of using thin-tailed (exponential) and fat-tailed (algebraic) kernels. Our numerical investigations show that the choice of kernel and boundary conditions have very limited impact when the nonlocal operator is close to the Laplacian, i.e. when the nonlocal Gray-Scott model is close to its local counterpart. However, when the spread of the nonlocal operator is large (i.e. exponential kernel with $\sigma=3.4$ and algebraic kernel with $a=0.42$), we showed that both, the type of kernel used and the type of boundary conditions considered, impact the profiles of pulse solutions. For instance, the exponential kernel leads to  \emph{mesa}-profiles while the algebraic kernel generates \emph{cat-ear} profiles. 
In addition, nonlocal Dirichlet boundary constraints lead to more prominent \emph{mesa} and \emph{cat-ear} profiles when compared to nonlocal Neumann and 'free' boundary constraints. Our simulations also show that increasing the size of the computational domain has a small effect in the overall shape of the pattern, generating shorter pulses, as well as allowing solutions to approach more closely the expected limit $(u,v) = (1,0)$ in the far field.

Similarly, when comparing periodic boundary conditions versus nonlocal boundary constraints, we find that 
the kernel's type and the boundary conditions used have minimal impact when the nonlocal operator behaves like the Laplacian. On the other hand, when periodic boundary constraints are applied, kernels with a large spread lead to pulse solutions with oscillations, 
while nonlocal boundary constraints generate \emph{mesa} and \emph{cat-ear} profiles with sharper corners. 
Thus, we conclude that the type of kernel and boundary constraints play a crucial role in the formation of pulse solution for the nonlocal Gray-Scott model, only when these kernels are wide.

Finally, we showed that when the nonlocal operator is defined using the exponential kernel the stationary nonlocal Gray-Scott model is equivalent to a system of quasilinear PDEs. We then give a heuristic argument that can be used to prove existence of smooth pulse solutions using perturbation methods and Lyapunov-Schmidt reduction. We also note that if the solutions satisfy certain assumptions, the system can further be simplified to a set of semilinear equations, thus partially justifying the emergence of pulses with \emph{mesa} profiles and providing a path for proving their existence using comparison principles. 
We plan to address the existence of these novel patterns in a future paper. In addition for future numerical investigations, we also plan to study the impact of boundary conditions on periodic patterns and extend our one-dimensional algorithm to problems set up on planar domains, i.e. $\Omega \subset \mathbbm{R}^2$.



\section{Appendix}

We present here the two tables that summarize the convergence tests discussed in Section \ref{s:continuation_cvg}. The Table \ref{tab:kernel_exp_alg_neumann_AB2_L1_H_eq_30_DT} displays the $L^1$ error and order of convergence obtained for setups with nonlocal Neumann boundary conditions. The Table \ref{tab:kernel_algebraic_dirichlet_real_AB2_L1_H_eq_30_DT} displays the $L^1$ error and order of convergence obtained for setups with either nonlocal Dirichlet or 'free' boundary conditions. This last table focuses on the setups where the kernel has a large spread ($\sigma=3,4$ and $a=0.42$) as the resulting solutions present more complex structures.

\begin{table}
\centering
\begin{tabular}{  m{1.7cm} m{1.5cm} m{1cm} m{1.7cm} m{1.5cm} m{1.7cm} m{1.5cm}}
 \multicolumn{7}{c}{\bf \Large Continuation: Convergence of Neumann Problem}\\[1ex]
 \multicolumn{7}{c}{\bf Exponential Kernel, $\sigma = 3.4$}\\[1ex]
\hline \hline
 dt & h & M & Error $u$ & Order $u$ & Error $v$ &Order $v$\\
 \hline \hline
    0.0025&      0.14648&     $2^9$&     2.96E-3&    - &       2.03E-1&    - \\
    0.00125&     0.073242&    $2^{10}$&    7.39E-4&      2.00&      8.03E-2&    1.34\\
    0.000625&     0.036621&    $2^{11}$&    1.82E-4&     2.02&      2.90E-2&    1.47\\
    0.0003125&     0.018311&    $2^{12}$&    7.43E-5&     1.30&     9.29E-3&    1.64\\
    0.00015625&    0.0091553&    $2^{13}$&    2.53E-6&      4.87&    1.11E-6&    13.03\\
\hline\\[2ex]
 \multicolumn{7}{c}{\bf Exponential Kernel, $\sigma = 4.0$}\\
\hline \hline
 dt & h & M & Error $u$ & Order $u$ & Error $v$ &Order $v$\\
 \hline \hline
    0.0025&      0.14648&      $2^{9}$&     3.70E-3&    - &      3.53E-2&    - \\
    0.00125&     0.073242&     $2^{10}$&    9.58E-4&    1.95&     6.59E-3&       2.42\\
    0.000625&     0.036621&     $2^{11}$&     1.10E-4&    3.12&     1.42E-3&     2.21\\
    0.0003125&     0.018311&     $2^{12}$&    5.07E-6&     4.44&    2.38E-6&     9.22\\
    0.00015625&    0.0091553&     $2^{13}$&    1.71E-6&    1.57&    7.93E-7&      1.59\\
\hline\\[2ex]
 \multicolumn{7}{c}{\bf Algebraic Kernel, $a = 0.39$}\\[1ex]
\hline \hline
 dt & h & M & Error $u$ & Order $u$ & Error $v$ &Order $v$\\
 \hline \hline
    0.0025&      0.14648&      $2^{9}$&     3.08E-3&    - &      4.77E-2&    - \\
    0.00125&     0.073242&     $2^{10}$&    7.69E-4&     2.00&     9.35E-3&     2.35\\
    0.000625&     0.036621&     $2^{11}$&    9.36E-5&     3.04&     1.70E-3&     2.46\\
    0.0003125&     0.018311&     $2^{12}$&    6.07E-6&     3.95&    9.27E-6&     7.52\\
    0.00015625&    0.0091553&     $2^{13}$&    1.54E-6&     1.98&    2.80E-6&     1.73\\
\hline\\[2ex]
 \multicolumn{7}{c}{\bf Algebraic Kernel, $a = 0.42$}\\
\hline \hline
 dt & h & M & Error $u$ & Order $u$ & Error $v$ &Order $v$\\
 \hline \hline
    0.0025&      0.14648&      $2^{9}$&     2.75E-3&    - &      8.10E-2&    - \\
    0.00125&     0.073242&     $2^{10}$&    6.00E-4&       2.20&     2.75E-2&     1.56\\
    0.000625&     0.036621&     $2^{11}$&    1.51E-4&       1.99&     2.09E-2&    0.39\\
    0.0003125&     0.018311&     $2^{12}$&    7.30E-5&       1.05&    8.66E-3&     1.27\\
    0.00015625&    0.0091553&     $2^{13}$&    2.04E-6&       5.16&    3.15E-6&     11.43\\
  \end{tabular}
\caption{Order of convergence based on $L^1$ error for Algorithm \ref{al:neumann} using the exponential and algebraic kernels with domain size $L = 75/2$, second-order Adams-Bashforth with $dt\sim h/30$, and boundary constraints representing the Neumann problem.}
\label{tab:kernel_exp_alg_neumann_AB2_L1_H_eq_30_DT}
\end{table}

\begin{table}
\centering
\begin{tabular}{  m{1.7cm} m{1.5cm} m{1cm} m{1.7cm} m{1.5cm} m{1.7cm} m{1.5cm}}
 \multicolumn{7}{c}{\bf \Large Continuation: Dirichlet and Real Line Problems}\\[1ex]
 \multicolumn{7}{c}{\bf Exponential Kernel, Dirichlet Problem, $\sigma = 3.4$}\\[1ex]
\hline \hline
 dt & h & M & Error $u$ & Order $u$ & Error $v$ &Order $v$\\
\hline \hline
    0.0025&     0.073242&      $2^{9}$&    4.34E-4 &    - &      2.71E-2    & - \\
    0.00125&     0.036621&     $2^{10}$&    2.28E-4 &   0.93&      1.30E-2  &   1.06\\
    0.000625&     0.018311&     $2^{11}$&    8.90E-5 &    1.36&     4.28E-3&     1.60\\
    0.0003125&    0.0091553&     $2^{12}$&    3.49E-5 &    1.35&     2.34E-3&    0.87\\
    0.00015625&    0.0045776&     $2^{13}$&    6.71E-7 &    5.70&    7.60E-7 &    11.59\\
\hline\\[2ex]
 \multicolumn{7}{c}{\bf Exponential Kernel, Real Line Problem, $\sigma = 3.4$}\\
\hline \hline
 dt & h & M & Error U & Order U & Error V &Order V\\
 \hline \hline
    0.0025  &   0.073242 &     $2^{9}$  &   1.85E-3   &  - &    7.67E-2   &  - \\
    0.00125 &    0.036621 &    $2^{10}$ &   6.05E-4  &   1.61&     2.79E-2   &  1.46 \\
    0.000625&     0.018311&     $2^{11}$&    1.72E-4 &    1.82&    9.19E-3  &   1.60 \\
    0.0003125&    0.0091553&     $2^{12}$&    2.69E-5 &    2.67&    1.33E-3  &   2.79  \\
    0.00015625&    0.0045776&     $2^{13}$&    1.44E-5 &   0.90&    1.23E-3  &  0.11 \\
\hline\\[2ex]
\multicolumn{7}{c}{\bf Algebraic Kernel, Dirichlet Problem, $a = 0.42$}\\
\hline \hline
dt & h & M & Error U & Order U & Error V &Order V\\
\hline \hline
    0.0025&     0.073242&      $2^{9}$&    7.29E-4&    - &   0.071183  &  - \\
    0.00125&     0.036621&     $2^{10}$&    3.11E-4&     1.23&     6.41E-2  &  0.15\\
    0.000625&     0.018311&     $2^{11}$&    1.90E-4&    0.72&     3.18E-2 &    1.01\\
    0.0003125&    0.0091553&     $2^{12}$&     9.01E-5&     1.07&     1.39E-2&     1.20\\
    0.00015625&    0.0045776&     $2^{13}$&     3.15E-5&     1.52&    4.66E-3&      1.58\\
\hline\\[2ex]
\multicolumn{7}{c}{\bf Algebraic Kernel,  Real Line Problem, $a = 0.42$}\\
\hline \hline
dt & h & M & Error U & Order U & Error V &Order V\\
\hline \hline
    0.0025&     0.073242&      $2^{9}$&     1.56E-3&     -&     2.57E-2   &  -\\
    0.00125&     0.036621&     $2^{10}$&    3.79E-4&     2.04&     1.93E-2   & 0.41\\
    0.000625&     0.018311&     $2^{11}$&     9.85E-5&     1.95&    8.33E-3  &   1.22\\
    0.0003125&    0.0091553&     $2^{12}$&    2.84E-5&     1.80&    1.31E-3 &    2.67\\
    0.00015625&    0.0045776&     $2^{13}$&    1.66E-5&    0.78&    1.23E-3&    0.096\\ 
\end{tabular}
\caption{Order of convergence based on $L^1$ error for  Algorithm \ref{al:dirichlet} using the exponential and algebraic kernel with domain size $L = 75/2$, second-order Adams-Bashforth with $dt\sim h /30$, and boundary constraints representing the Dirichlet and Real Line problems.}
\label{tab:kernel_algebraic_dirichlet_real_AB2_L1_H_eq_30_DT}
\end{table}

\bibliographystyle{plain}
\bibliography{nonlocal_bib}

\end{document}